\documentclass[12pt, twoside, letterpaper]{amsart}\usepackage{amsmath, hyperref, color, amssymb}\usepackage{srcltx}\usepackage{geometry}\newtheorem{thm}{Theorem}[section]\newtheorem{lem}[thm]{Lemma}\theoremstyle{definition}\newtheorem{as}[thm]{Assumption}\theoremstyle{remark}\newtheorem{rem}[thm]{Remark}\newtheorem{exa}[thm]{Example}\numberwithin{equation}{section}\newcommand{\R}{\mathbb{R}}\newcommand{\PP}{\mathbb{P}}\newcommand{\sign}{\mathsf{sign}}\newcommand{\dbracc}[1]{[\kern-0.15em[ #1 ]\kern-0.15em]}\newcommand{\dbraco}[1]{[\kern-0.15em[ #1 [\kern-0.15em[}\newcommand{\dbraoc}[1]{]\kern-0.15em] #1 ]\kern-0.15em]}\newcommand{\dbraoo}[1]{]\kern-0.15em] #1 [\kern-0.15em[} \newcommand{\be}{\begin{equation}}\newcommand{\ee}{\end{equation}}\newcommand{\ba}{\begin{aligned}}\newcommand{\ea}{\end{aligned}}\renewcommand{\P}{\PP}\newcommand{\Pas}{\text{$\P$--a.s.}}

\begin{document}

\title[Asymptotic analysis of the expected utility maximization problem]
{Asymptotic analysis of 
the expected utility maximization problem with respect to perturbations of the num\'eraire}

\begin{abstract}
In an incomplete model, where  under an appropriate num\'eraire, the stock price process is driven by a sigma-bounded  semimartingale, we investigate the behavior of the expected utility maximization problem under small perturbations of the num\'eraire. 
We establish a quadratic approximation of the value function and a first-order expansion of the terminal wealth. Relying on a description of the base return process in terms of its semimartingale characteristics, we also construct wealth processes and nearly optimal strategies that allow for matching the primal value function up to the second order. 
We also link perturbations of the num\'eraire to distortions of the finite-variation part and martingale part of the stock price return and characterize the asymptotic expansions in terms of the risk-tolerance wealth process. 
\end{abstract}

\author{Oleksii Mostovyi}\thanks{The author would like to thank Mihai S\^ irbu for discussing topics related to the core of the paper. The author has been supported by the National Science Foundation under grants No. DMS-1600307 (2015 - 2019) and No. DMS-1848339 (2019 - 2024). Any opinions, findings, and conclusions or recommendations expressed in this material are those of the author's and do not necessarily reflect the views of the National Science Foundation.}

\address{Oleskii Mostovyi, Department of Mathematics, University of Connecticut, Storrs, CT 06269, United States}
\email{oleksii.mostovyi@uconn.edu}%


\date{\today}%

\maketitle

\section{Introduction}

In the settings of a complete financial market, it is proven in \cite{GemKarRoc95}  that the choice of a num\'eraire affects neither arbitrage-free prices of the securities nor replicating strategies (see also a discussion in \cite{HobHen}). However, by an appropriate change of num\'eraire (sometimes combined with a change of measure), one can simplify a valuational framework, see, e.g., \cite{GemKarRoc95}. Possibly the most illuminating example corresponds to the LIBOR market interest rate model, which is based on a dynamic change of num\'eraire and which allows for pricing a wide class of interest rate derivatives.

In incomplete markets, the situation is more delicate in general. As num\'eraire is a crucial ingredient in essentially all problems of mathematical finance, it is important to understand their sensitivity to misspecifications of the num\'eraire. In this paper, in a general incomplete semimartingale model of a financial market, we investigate the response of the value function and the optimal solution to the expected utility maximization from terminal wealth problem to small perturbations of the num\'eraire.  To the best of our knowledge, asymptotics for the expected utility maximization problem to perturbations of num\'eraire has not been studied in the literature. We establish a second-order expansion of the value function, a first-order approximation of the terminal wealth, and construct wealth processes and corrections to optimal strategies that allow for the second-order matching of the primal value  function. The latter development is conducted via a representation of the base return process in terms of its semimartingale characteristics. In particular, we establish an envelope-type theorem for both primal and dual value functions. We also characterize  the asymptotic expansions via the risk-tolerance wealth process, provided that the later exists, and give a characterization of the correction terms via a Galtchouk-Kunita-Watanabe decomposition under certain changes of measure and num\'eraire. Note that the risk-tolerance wealth process was introduced in \cite{KS2006b}.

Our results provide a way to estimate the effect of misspecification of the initial data on the expected utility maximization problem. This in particular applies to models, which allow for explicit solutions, see e.g., \cite{Thaleia01}, \cite{KallsenLog}, \cite{imHuMuller05}, \cite{KS2006b}, \cite{GR12}, \cite{Horst2014}, \cite{Santacroce2}, and to so-called  asymptotically complete models, see \cite{Scott17, RSA17}. 
In many cases,  a closed-form solution ceases to exist under perturbations of model parameters.  Note that \cite{KS2006b}, \cite{Horst2014}, and \cite{Santacroce2} deal with a general utility function. This, in particular, emphasizes the importance of non confining oneself to power or logarithmic~utilities.

In order to obtain the asymptotic expansions mentioned above, we introduce a {\it linear} parametrization of {\it returns} of a perturbed family of num\'eraires such that 
the corresponding num\'eraires are positive wealth processes for the values of the parameter being sufficiently close to $0$. Note that positivity  
is a {necessary} condition for a process to be considered a num\'eraire.  
 Even though, in principle, by a num\'eraire one can choose any strictly positive semimartingale, in this work, we focus on tradable num\'eraires, in the terminology of \cite{Becherer}, i.e., the ones can be obtained as  outcomes of trading strategies. Such a choice is standard in the mathematical finance literature, see for example \cite{Becherer}, \cite{KS2006}, \cite{KS2006b}, \cite{KarKar07}.

In the case when the stock price process is one-dimensional and continuous, our structure of perturbations is closely connected to distortions  of the finite-variation part of the return of the stock (as in \cite{MostovyiSirbuModel}) and perturbations of the volatility (as in \cite{herJohSei17}), see the discussion in section \ref{8241a} below. 
The proofs 
rely on the 
auxiliary minimization problems, which in turn are closely related to the ones in \cite{Pham98}, \cite{phamSchw},  \cite{laurenPham}, \cite{KalCher07}, \cite{CzichSchweiz2013}, \cite{Santacroce1}, see also an overview of several approaches to quadratic problems in \cite{PhamBook}. 
Asymptotics analysis based on Malliavin calculus is implemented in \cite{Monoy2013}. 
Simultaneous primal-dual asymptotic expansion method in mathematical finance has been (arguably)  introduced in \cite{vicky02} in the context of a utility-based pricing problem. Related analysis has been performed (at approximately the same time) in \cite{henHob02}, \cite{Kallsen02}. The first-order differentiability of the value functions  with respect to the perturbations of the initial wealth  and convergence of the optimizers are established in \cite{KS}, whereas twice-differentiability is investigated~in~\cite{KS2006}. 

As we expand the value function also in the initial wealth, analysis from \cite{KS2006} turns out to be very helpful in the present work. On the other hand, Remark \ref{rem:DeltaXStrategies} below gives {\it corrections to the optimal trading strategy}, such that the corresponding wealth processes match the indirect utility up to the second order. This complements the results in \cite{KS2006}. In this part, a representation of the base return process and in terms of its semimartingale characteristics is crucial. 

The closest paper (to the best of our knowledge), mathematically, is \cite{MostovyiSirbuModel}, which deals with different perturbations, namely of  the market price of risk, and where the underlying framework includes a continuous and one-dimensional stock price process. In the present paper, we impose neither one-dimensionality nor continuity of the stock (and the perturbations are different from the ones in \cite{MostovyiSirbuModel}).

The remainder of this paper is organized as follows. In section \ref{secModel}, we present the model, in section \ref{secExpansionTheorems} we formulate auxiliary minimization problems and state the expansion theorems; section \ref{approxTradingStrategies}, contains an explicit construction of nearly optimal wealth processes and corrections to the optimal strategies that allow for matching the primal value function up to the second order. In section \ref{secProofs},  we give proofs of these results. In section \ref{secRiskTol}, we relate the expansion theorems to the existence of a risk-tolerance wealth process, and we conclude the paper with section \ref{secCounterexamples}, where we show the necessity of Assumptions \ref{boundedJumps} and \ref{integrabilityAssumption}, under which the expansion theorems are proven.

\section{Model}\label{secModel}
\subsection{Parametrized family of stock prices processes}
Let us consider a complete stochastic basis $\left(\Omega, \mathcal F, \{\mathcal F_t\}_{t\in[0,T]}, \mathbb P\right)$, where $T\in (0,\infty)$ is the time horizon, $\mathcal F$ satisfies the usual conditions, and $\mathcal F_0$ is a trivial $\sigma$-algebra. For the $0$-model, we assume that there are  $d$ traded stocks, whose returns are modeled via a general $d$-dimensional semimartingale $(\rho^1, \dots, \rho^d)$, as well as a bank account, whose price is equal to $1$ at all times. We set $R = (0,\rho^1,\dots, \rho^d)$ and suppose that (every component of) $R_0 =0$.

The num\'eraire of $0$-model is $N^0 \equiv 1$, equivalently the num\'eraire, whose return equals to zero and whose initial value equals $1$. For perturbed models, we introduce linear perturbations of the returns of the num\'eraires, which are given by 

\begin{equation}\label{numReturn}
\varepsilon\theta \cdot {R}, \quad \varepsilon \in(-\varepsilon_0, \varepsilon_0),
\end{equation}
where $\theta$ is some predictable and ${R}$-integrable process \textcolor{black}{that represents the {\it proportions of a wealth process} invested in the corresponding stocks for some portfolio (i.e., $\theta^0_t = 1 - \sum\limits_{i=1}^d \theta^i_t$, $t\in[0,T]$) and that satisfies Assumptions \ref{boundedJumps} and \ref{integrabilityAssumption} below,} and $\varepsilon_0$ is a positive constant specified via Assumption \ref{boundedJumps}. Equivalently, \eqref{numReturn} can be restated in terms of the parametrized family of num\'eraires $(N^\varepsilon)_{\varepsilon\in(-\varepsilon_0, \varepsilon_0),}$ that satisfy 
\begin{equation}\label{numeraire}
N^\varepsilon = \mathcal E\left((\varepsilon\theta) \cdot {R}\right),\quad \varepsilon \in(-\varepsilon_0, \varepsilon_0),
\end{equation}
where $\mathcal E$ denotes the stochastic exponential. 
Thus, the family of stock price processes under num\'eraires $N^\varepsilon$ is given by
\begin{equation}\nonumber
S^\varepsilon := \left( \frac{1}{N^\varepsilon}, \frac{\mathcal E(\rho^1)}{N^\varepsilon},\dots,\frac{\mathcal E\left(\rho^d\right)}{N^\varepsilon}\right), \quad \varepsilon \in(-\varepsilon_0, \varepsilon_0).
\end{equation}

\subsection{Primal problem}
Let $U$ be a utility function satisfying Assumption \ref{rra} below.
\begin{as}\label{rra}
The function 
$U$: $(0,\infty)\to\mathbb R$ is strictly increasing, strictly concave, two times continuously differentiable, and is such that for some positive  constants $c_1$ and $c_2$, we have 
\begin{equation}\nonumber 
{c_1} \leq A(x) := -\frac{U''(x)x}{U'(x)} \leq c_2.
\end{equation}
\end{as}
The admissible wealth processes are given by
\begin{equation}\nonumber 
\mathcal X(x,\varepsilon) := \left\{ 
x + H\cdot S^{\varepsilon}\geq 0:~H~{\rm is}~S^\varepsilon-{\rm integrable}
\right\},\quad (x,\varepsilon)\in[0, \infty)\times(-\varepsilon_0, \varepsilon_0).
\end{equation}
The primal value functions (parametrized by $\varepsilon$) are given by
\begin{equation}\label{primalProblem}
u(x, \varepsilon) := \sup\limits_{X\in\mathcal X(x,\varepsilon)} \mathbb E\left[ U(X_T)\right],\quad (x,\varepsilon)\in(0, \infty)\times (-\varepsilon_0, \varepsilon_0).
\end{equation}
We use the convention
\begin{displaymath}
\mathbb E\left[ U(X_T)\right] := -\infty, \quad {\rm if} \quad 
\mathbb E\left[ U^{-}(X_T)\right] = \infty, 
\end{displaymath}
where $U^{-}$ is the negative part of $U$. 
\subsection{Dual problem}Analysis of \eqref{primalProblem} is performed via the dual problem. 
As usually, let us  set 
\begin{equation}\label{dualDomain}
\begin{split}
\mathcal Y(y,\varepsilon):= \left\{
Y\geq 0:\quad\right.& Y~{\rm is~a~supermartingale~starting~at~}y,\\
& {\rm and~such~that} XY = \left(X_tY_t\right)_{t\geq 0}~{\rm is~a~supermartingale}\\
& {\rm for~each}~X\in~\mathcal X(1, \varepsilon)\left.\right\},  
\quad (y,\varepsilon)\in[0, \infty)\times (-\varepsilon_0, \varepsilon_0).
 \end{split}
\end{equation}
\begin{rem}
Definition \eqref{dualDomain} is an alternative version of a two-step natural definition of the dual domain, where in the first step one defines $\mathcal Y(y,0)$ as above and then sets 
$\mathcal Y(y,\varepsilon) = \mathcal Y(y,0) N^\varepsilon$. However, Lemma \ref{keyConcreteCharLem} asserts that both constructions are equivalent.
\end{rem}
We set the convex conjugate to $U$ as
\begin{equation}\nonumber 
V(y) := \sup\limits_{x > 0} \left( U(x) - xy\right),\quad  y\in (0,\infty).
\end{equation}
Let us recall that for every $x>0$ with $y = U'(x)$, we have 
$$U''(x)V''(y) = -{1}.$$
Setting $$B(y) := -\frac{V''(y)y}{V'(y)}= \frac{1}{A(x)},$$
we deduce from Assumption \ref{rra} that
$$\frac{1}{c_2} \leq B(y)  \leq \frac{1}{c_1}, \quad y> 0.$$
The corresponding dual value functions are set as
\begin{equation}\label{dualProblem}
v(y, \varepsilon) := \inf\limits_{Y\in\mathcal Y(y, \varepsilon)}\mathbb E\left[ 
V\left(Y_T\right)
\right], \quad (y,\varepsilon)\in(0, \infty)\times (-\varepsilon_0, \varepsilon_0).
\end{equation}
With  $V^{+}$ denoting the positive part of $V$, if 
\begin{displaymath}
\mathbb E\left[ 
V^{+}\left(Y_T\right)
\right] = \infty \quad {\rm we~set}\quad
\mathbb E\left[ 
V\left(Y_T\right)
\right] := \infty
\end{displaymath}

\subsection{Technical assumptions}
For nondegeneracy of $0$-model, we suppose that
\begin{equation}\label{finCond}
 {\rm there~exists}\quad x>0\quad {\rm such~that}\quad u(x,0) < \infty.
\end{equation}

One needs to ensure that the perturbations of the form \eqref{numReturn} (or equivalently in the form \eqref{numeraire}) are such that the resulting processes $N^\varepsilon$ are nonnegative at least for $\varepsilon$ being sufficiently close to $0$, as a necessary way of making $N^\varepsilon$'s num\'eraires. This can be achieved via the following condition. Example \ref{counterExBoundedJumps} below demonstrates the necessity of a boundedness Assumption \ref{boundedJumps}.
\begin{as}\label{boundedJumps}
We  
suppose that 
there exists $\varepsilon_0>0$ such that the jumps of the process $\bar R:= -\theta\cdot {R}$ are bounded by $\tfrac{1}{2\varepsilon_0}$, i.e., 
$$|\Delta \bar R_t|\leq \tfrac{1}{2\varepsilon_0},\quad t\in[0,T].$$
\end{as}
Note that Assumption \ref{boundedJumps} implies that $N^\varepsilon$ in \eqref{numeraire} is a strictly positive process $\Pas$, for every $\varepsilon\in(-\varepsilon_0 ,\varepsilon_0)$.

\subsection{Absence of arbitrage}
The absence of arbitrage opportunities in the $0$-model in the sense of no unbounded profit with bounded risk follows from condition \eqref{NUPBR}, which by the results of \cite{KarKar07}  can equivalently be stated as 
\begin{equation}\label{NUPBR}
\mathcal Y(1, 0)~{\rm contains~a~strictly~positive~element}.
\end{equation}
\textcolor{black}{We refer to \cite{KarKar07} for characterizations of no unbounded profit with bounded risk condition, which is also equivalent to the existence of a strict sigma-martingale density, see  \cite{TakaokaSchweizer} for details. }
\begin{rem}
Condition \eqref{NUPBR} and Lemma \ref{keyConcreteCharLem} imply no unbounded profit with bounded risk for every $\varepsilon\in(-\varepsilon_0, \varepsilon_0)$, thus 
$$\mathcal Y(1, \varepsilon)\neq \emptyset,\quad \varepsilon\in(-\varepsilon_0, \varepsilon_0).$$
\end{rem}
\begin{rem}\label{rem4221} 
Assumption \ref{rra} implies that $U$ satisfies the Inada conditions and that asymptotic elasticity of $U$ (in the sense of \cite{KS}) is less than $1$, see \cite[Lemma 3]{KS2006} for the proof. Therefore, 
under \eqref{rra}, \eqref{finCond}, and \eqref{NUPBR}, existence and uniqueness of a solution to \eqref{primalProblem} for every $x>0$  and other standard assertions of the utility maximization theory follow from the abstract theorems in \cite{KS}.
\end{rem}
\begin{rem}
\cite[Theorem 2.1]{KabKarSong16} gives a characterization of no unbounded profit with bounded risk condition in terms of the existence of {\it local martingale} deflators (as opposed to  {\it supermartingale} deflators in \cite{KarKar07}). 
\end{rem} 
For every $x>0$, under Assumption \ref{rra}, \eqref{finCond}, and \eqref{NUPBR} it follows from Remark \ref{rem4221}, that $y = u_x(x,0)$ exists and is unique and there exist unique solutions to \eqref{primalProblem} and \eqref{dualProblem}, $\widehat X(x,0)$ and $\widehat Y(y,0)$, respectively, \textcolor{black}{such that $\widehat X(x,0)\widehat Y(y,0)$ is a uniformly integrable martingale under $\mathbb P$}.
For $x>0,$ and with $y=u_x(x,0)$ we define a probability measure $\mathbb R$ via 
\begin{equation}\label{measureR}
\frac{d \mathbb R(x)}{d\mathbb P} := \frac{\widehat X_T(x,0) \widehat Y_T(y,0)}{xy}.
\end{equation}
Note that,  $\mathbb R(x)$ defined in \eqref{measureR} coincides with the measure $\mathbb R(x)$ in the notations of \cite{KS2006}, \cite{KS2006b} and with measure $\mathbb R(x,0)$ in terminology of \cite{MostovyiSirbuModel},  \textcolor{black}{and that $\mathbb R(x)$ naturally appears in the asymptotic analysis of optimal investment, see \cite{KS2006}, \cite{KS2006b}, and \cite{MostovyiSirbuModel}.} 

Since we consider an expansion also in the initial wealth, in order for the value function $u$ to be twice differentiable in the first argument (which corresponds to the initial wealth $x$), we need to impose the sigma-boundedness assumption, see \cite[Definition 1]{KS2006} for the definition, also \cite{KS2006b} and \cite{MihaiSara} contain discussions on this subject and applications of sigma-bounded processes to the problem of the expected utility maximization.
 \begin{as}\label{sigmaBoundedness}
Let $x>0$ be fixed. We suppose that 
the process 
\begin{displaymath}
S^{\widehat X(x,0)} := \left( \frac{x}{\widehat X(x,0)}, \frac{x\mathcal E(\rho^1)}{\widehat X(x,0)}, \dots,\frac{x\mathcal E(\rho^d)}{\widehat X(x,0)}\right)
\end{displaymath}
is sigma-bounded.
\end{as}
\textcolor{black}{When using $S^{\widehat X(x,0)}$, we discount the assets by the normalized primal optimizer for the $0$-model.} We also need the following integrability assumption on perturbations, whose necessity is demonstrated in Example \ref{8231} below.
\begin{as}\label{integrabilityAssumption}
Let $x>0$ be fixed.
There exists $c>0$, such that 
$$ \mathbb {E}^{\mathbb R(x)}\left[\exp\left\{c\left(\left|\bar R_T\right| + [ \bar R, \bar R]_T \right)\right\} \right]<\infty.$$
\end{as}

\section{Expansion Theorems}\label{secExpansionTheorems}
We begin with an envelope theorem.  
\begin{thm}\label{mainThm1}
Let $x>0$ be fixed, assume that \eqref{finCond} and \eqref{NUPBR} as well as Assumptions \ref{rra}, \ref{boundedJumps}, \ref{sigmaBoundedness},  and \ref{integrabilityAssumption} hold, and let $y = u_x(x,0)$. 
 Then there exists $\bar \varepsilon>0$ such that for every $\varepsilon \in(-\bar \varepsilon, \bar \varepsilon)$,
$u(\cdot,\varepsilon)$ and 
$v(\cdot,\varepsilon)$ are finite-valued functions. 
The functions $u$ and $v$ are jointly differentiable (and, consequently, continuous) at $(x,0)$ and $(y,0)$, respectively. We also have  
\begin{equation}\label{gradientC}
\nabla u (x,0)=\begin{pmatrix}y\\ u_{\varepsilon}(x,0)\end{pmatrix}\quad and \quad \nabla v (y,0)=\begin{pmatrix}-x\\v_{\varepsilon}(y,0)\end{pmatrix},
\end{equation} where
\begin{equation}\nonumber 
u_\varepsilon(x,0) =v_\varepsilon(y,0)= xy\mathbb E^{\mathbb R(x)}\left[ \bar R_T\right]. 
\end{equation}

\end{thm}
\textcolor{black}{Note that, the key formula in Theorem \ref{mainThm1} is the expression for $u_\varepsilon(x,0).$
In the case when $\hat Y(y,0)$ is a uniformly integrable martingale itself, this process is often used to define a new measure ${\hat {\mathbb Q}(y)}$ via $\frac{d{\hat {\mathbb Q}(y)}}{d\mathbb P}:= \frac{\hat Y(y,0)}{y}$. Then, the first-order derivatives in $\varepsilon$ can be restated as 
$$u_\varepsilon(x,0) =v_\varepsilon(y,0)= y\mathbb E^{\hat{\mathbb Q}(y)}\left[ \hat X_T(x,0)\bar R_T\right].$$}

For a given $x>0$, in order to obtain the second-order terms in quadratic expansions of the value functions, we introduce $\mathbf H^2_0(\mathbb R(x))$, the space of square integrable martingales under $\mathbb R(x)$ that start at $0$. We recall that $S^{X(x, 0)}$ was defined in Assumption \ref{sigmaBoundedness} and, with $y= u_x(x,0)$,  we set
\begin{displaymath}
\begin{split}
\mathcal M^2(x, 0) &:= \left\{M\in\mathbf H_0^2(\mathbb R(x)):  M = H\cdot S^{\widehat X(x, 0)} \right\},\\
\mathcal N^2(y, 0) &:= \left\{N\in\mathbf H_0^2(\mathbb R(x)):  MN~{\rm is}~\mathbb R(x)-{\rm martingale~for~every}~M\in\mathcal M^2(x, 0)\right\}.
\end{split}
\end{displaymath}
\subsection*{Auxiliary optimization problems} 
Recalling that $A$ and $B$ denote the  the relative risk aversion  and the relative risk tolerance of $U$, respectively,  
following \cite{KS2006}, for a fixed $x>0$ (with $y = u_x(x,0)$), we set 
\begin{equation}\label{axxC}
a(x,x) := \inf\limits_{M \in\mathcal M^2(x, 0)}\mathbb E^{\mathbb R(x)}\left[ A(\widehat X_T(x,0))(1 + M_T)^2\right],
\end{equation}
\begin{equation}\label{byyC}
b(y,y) := \inf\limits_{N \in\mathcal N^2(y, 0)}\mathbb E^{\mathbb R(x)}\left[ B(\widehat Y_T(y,0))(1 + N_T)^2\right].
\end{equation}
 We refer to \cite{KS2006} for the details behind the derivation of \eqref{axxC} and \eqref{byyC}. Note that \eqref{axxC} and \eqref{byyC} govern the second-order derivatives of $u$  in $x$ and $v$ in $y$, respectively. 
\begin{rem}
Existence and uniqueness of a solution to every quadratic minimization problem  in this paper follows from the closedness of its domain (in the appropriate sense), convexity of the objective, and Komlos' lemma, see  \cite[Lemma 2]{KS2006}.
\end{rem}

Let us also set
\begin{equation}\label{defFG}
F:= \bar R_T\quad {\rm and }\quad  G:= [{\bar R}, {\bar R}]_T.
\end{equation}
We consider the following minimization problems:
\begin{equation}\label{addC}
\begin{array}{rcl}
a(\varepsilon, \varepsilon) &:=& \inf\limits_{M\in\mathcal M^2(x, 0)}\mathbb E^{\mathbb R(x)}\left[A(\widehat X_T(x,0))(M_T + xF)^2 - 
2xFM_T - x^2(F^2 + G) \right],\\ 
\end{array}
\end{equation}
\begin{equation}\label{bddC}
\begin{array}{rcl}
b(\varepsilon, \varepsilon) &:=& \inf\limits_{N\in\mathcal N^2(y, 0)}\mathbb E^{\mathbb R(x)}\left[B(\widehat Y_T(y,0))(N_T - yF)^2 + 
2yFN_T - y^2(F^2 - G) \right].\\ 
\end{array}
\end{equation}
\textcolor{black}{Quadratic minimization problems \eqref{addC} and \eqref{bddC} govern the second-order correction terms associated with perturbations in $\varepsilon$ in the expansion for $u$ and $v$, where the exact structure is given through Theorem \ref{mainThm2}.} 
Let $M^1(x,0)$ and $N^1(y,0)$ designate the unique optimizers to \eqref{addC} and \eqref{bddC} respectively. Then, we define
\begin{equation}\label{axdC}
\begin{split}
a(x, \varepsilon) := \mathbb E^{\mathbb R(x)}&\left[- xF(1 + M^0_T(x,0))\right.\\
&\left.+ A(\widehat X_T(x,0))(xF + M^1_T(x,0))(1 + M^0_T(x,0)) \right],
\end{split}
\end{equation}
\begin{equation}\label{bydC}\begin{split}
b(y, \varepsilon) := \mathbb E^{\mathbb R(x)}&\left[yF(1 +N^0_T(y,0)) \right.\\
&\left.+ B(\widehat Y_T(y,0))(-yF + N^1_T(y,0))(1 + N^0_T(y,0)) \right].\end{split}
\end{equation}
Theorem  \ref{mainThm2} contains the quadratic expansions of the value functions. 
\begin{thm}\label{mainThm2}
Let $x>0$ be fixed.  Assume all  conditions of Theorem \ref{mainThm1} hold,  with $y = u_x(x,0)$.  
 Let us define 
\begin{equation}\label{12136C}
H_u(x,0) :=  -\frac{y}{x}\begin{pmatrix} 
      a(x,x) 		& a(x, \varepsilon)\\ 
  a(x, \varepsilon)	& a(\varepsilon, \varepsilon)\\ 
\end{pmatrix},
\end{equation}
where $a(x,x)$, $a(\varepsilon, \varepsilon)$, and $a(x, \varepsilon)$ are specified in \eqref{axxC}, \eqref{addC}, and \eqref{axdC},  and, respectively,
\begin{equation}\nonumber
H_v(y,0) := \frac{x}{y}\begin{pmatrix} 
      b(y,y) 		& b(y, \varepsilon)\\ 
  b(y, \varepsilon)	& b(\varepsilon, \varepsilon)\\ 
\end{pmatrix},
\end{equation}
where  $b(y,y)$, $b(\varepsilon, \varepsilon)$, $b(y, \varepsilon)$ are defined in \eqref{byyC}, \eqref{bddC}, and \eqref{bydC}.
Then, the   second-order expansions  around $(x,0)$ of $u$ is given by 
\begin{equation}\nonumber
u(x+\Delta  x,\varepsilon) = u(x,0) + (\Delta  x\quad \varepsilon) \nabla u(x,0) + \tfrac{1}{2}(\Delta  x\quad \varepsilon) 
H_u(x,0)
\begin{pmatrix}
\Delta  x\\
\varepsilon\\
\end{pmatrix} + o(\Delta  x^2 + \varepsilon^2),
\end{equation}
likewise, the  quadratic expansion  around $(y,0)$ of $v$ is 
\begin{equation}\nonumber
v(y+\Delta y,\varepsilon) = v(y,0) + (\Delta  y\quad \varepsilon) \nabla v(y,0) + \tfrac{1}{2}(\Delta  y\quad \varepsilon) 
H_v(y,0)
\begin{pmatrix}
\Delta  y\\
\varepsilon\\
\end{pmatrix} + o(\Delta  y^2 + \varepsilon^2).
\end{equation}\end{thm}
\begin{rem}\label{rem:415} Similarly to \cite{MostovyiSirbuModel}, slightly abusing the language and without necessarily having twice differentiability of $u$ and $v$, we call by $H_u(x,0)$ and $H_v(y,0)$ their respective Hessians. 
\end{rem}
Theorem \ref{162} gives a relationship between the auxiliary value functions as well as between the optimizers to auxiliary minimization problems \eqref{axxC}, \eqref{byyC}, \eqref{addC}, and \eqref{bddC}.
\begin{thm}\label{162}
Let $x>0$ be fixed, the assumptions of Theorem \ref{mainThm1} hold, and $y = u_x(x,0)$. Then, the auxiliary value functions satisfy
\begin{equation}\nonumber
\begin{pmatrix}
a(x,x) & 0\\
a(x, \varepsilon) &-\frac{x}{y} \\
\end{pmatrix}
\begin{pmatrix}
b(y,y) & 0\\
b(y, \varepsilon) & -\frac{y}{x} \\
\end{pmatrix} = 
\begin{pmatrix}
1 & 0\\
0 & 1 \\
\end{pmatrix}
\end{equation}
and
\begin{equation}\nonumber
\frac{y}{x}a(\varepsilon, \varepsilon) + \frac{x}{y}b(\varepsilon, \varepsilon) = a(x, \varepsilon)b(y, \varepsilon).
\end{equation}
The minimizers to 
auxiliary minimization problems \eqref{axxC}, \eqref{addC}, \eqref{byyC}, and \eqref{bddC} 
are related via the following formulas:
\begin{equation}\nonumber
\begin{split}
&\begin{pmatrix}
a(x,x) & 0 \\
a(x, \varepsilon) & -\frac{x}{y} 
\end{pmatrix}
\begin{pmatrix}
N^0_T(y,0) + 1\\
N^1_T(y,0) - yF
\end{pmatrix}= 
A\left(\widehat X_T(x,0)\right)
\begin{pmatrix}
M^0_T(x,0) + 1\\
M^1_T(x,0) + xF
\end{pmatrix},
\end{split}
\end{equation}
\begin{equation}\nonumber\begin{split}
&\begin{pmatrix}
b(y,y)& 0\\b(y, \varepsilon) &- \frac{y}{x}\end{pmatrix}
\begin{pmatrix}
1 + M^0_T(x,0) \\ xF + M^1_T(x,0)\\
\end{pmatrix} 
= 
B(\widehat Y_T(y,0))
\begin{pmatrix} 1 + N^0_T(y,0) \\ -yF + N^1_T(y,0)\\\end{pmatrix} .
\end{split}\end{equation}
Moreover, the product of any process in $\{\widehat X(x,0)M^0(x,0), \widehat X(x,0)M^1(x,0), \widehat X(x,0)\}$ and in $\{\widehat Y(y,0)N^0(y,0), \widehat Y(y,0)N^1(y,0), \widehat Y(y,0)\}$ is a $\mathbb P$-martingale.
\end{thm}

Theorem \ref{161} gives the derivatives of the optimizers via solutions to auxiliary minimization problems  \eqref{axxC}, \eqref{byyC}, \eqref{addC}, and \eqref{bddC}. 
\begin{thm}\label{161}

Let us suppose that $x>0$ is  fixed and the assumptions of Theorem \ref{mainThm1} hold.   Then, with 
\begin{displaymath}
\begin{split}
X^{x}_T(x, 0)&:= \frac{\widehat X_T(x,0)}{x} (1 + M^0_T(x,0))\quad and \quad  
X^{\varepsilon}_T(x,0):=  \frac{\widehat X_T(x,0)}{x} (xF + M^1_T(x,0) ),
\end{split}
\end{displaymath}
we have
\begin{equation}\label{12511}
\begin{split}
\lim\limits_{|\Delta  x| + |\varepsilon| \to 0}\frac{1}{|\Delta  x| + |\varepsilon|}\left|
\widehat X_T(x + \Delta  x, \varepsilon) -\widehat X_T(x,0)- \Delta  x X^{x}_T(x,0) -\varepsilon X^{\varepsilon}_T(x,0)\right|&= 0.
\end{split}
\end{equation}
Likewise, denoting $y = u_x(x,0)$ and  with 
\begin{displaymath}
\begin{split}
Y^{y}_T(y,0):= 
\frac{\widehat Y_T(y,0)}{y}(1 + N^0_T(y,0) )\quad and \quad 
Y^{\varepsilon}_T(y,0)&:= -
\frac{\widehat Y_T(y,0)}{y}(yF-N^1_T(y,0)),
\end{split}
\end{displaymath}
we have 
\begin{equation}\label{1252}
\begin{split}
\lim\limits_{|\Delta  y| + |\varepsilon |\to 0}\frac{1}{|\Delta  y| + |\varepsilon|}\left|
\widehat Y_T(y + \Delta  y, \varepsilon) - \widehat Y_T(y,0)-\Delta  y Y^{y}_T(y,0)-\varepsilon Y^{\varepsilon}_T(y,0)\right|&= 0,
\end{split}
\end{equation}
In both \eqref{12511} and \eqref{1252}, the convergence is in $\mathbb P$-probability.
\end{thm}

\section{Construction of nearly optimal wealth processes}
\label{approxTradingStrategies}

Here $x>0$ will be fixed  ${\widehat{\pi}}$ will denote the optimal proportion invested in stock for $0$-model and initial wealth $x$, i.e., ${\widehat{\pi}}$ satisfies
$$\widehat X(x,0)=x\mathcal E\left({\widehat{\pi}}\cdot R\right),$$
where  $R = (\rho^0,\rho^1,\dots, \rho^d)$, $\rho^0\equiv 0$.  For the results below, we will need a representation of $R$ in terms of its {\it predictable characteristics}. Notation-wise here, we follow \cite{JS}. Thus, we fix the {\it truncation function} $h(x):x\to x1_{\{|x|\leq 1\}}$ and denote by $R^c$ the {\it continuous martingale part} of $R$, by $B$ the {\it predictable finite variation part} of $R$ (corresponding to the truncation function $h$), by $\mu$ the {\it jump measure} of $R$, i.e., a random counting measure on $[0,T]\times \mathbb R^d$ defined by
\begin{equation}\nonumber
\mu\left([0,t]\times E\right) := \sum\limits_{0\leq s\leq t}1_{\{E\backslash\{0\}\}}(\Delta R_s),\quad t\in[0,T],E\subseteq \mathbb R^d,
\end{equation}
where $1_E$ is the indicator function of a set $E$, by $\nu$ we denote the {\it predictable compensator} of $\mu$, i.e., a {\it predictable random measure} on $[0,T]\times \mathbb R^d$, such that, in particular, $\left(x1_{\{|x|\leq 1\}}\right)*(\mu - \nu)$ is a {\it purely discontinuous local martingale}. Setting the {\it quadratic covariation process} $C:= [R^c,R^c]$ of $R^c$, we call $(B, C,\eta)$ the {\it triplet of predictable characteristics} of $R$ (associated with the truncation function $h$). 

 It it well-known (see for example \cite{JS}), that semimartingale $R$ can be represented in terms of $(B, C,\eta)$ as
\begin{equation}\nonumber
R = R^c + B + \left(x1_{\{|x|\leq 1\}}\right)*(\mu - \nu) + \left(x1_{\{|x|>1\}}\right)*\mu.
\end{equation}
Note that predictable characteristics $(B, C, \nu)$ are unique up to a $\mathbb P$-null set.
Moreover, let us define a predictable scalar-valued locally integrable increasing process process $A$ as
\begin{equation}\nonumber
A:= \sum\limits_{i \leq d} Var(B^i)+ \sum\limits_{i\leq d}C^{i,i} + (\min(1, |x|^2))*\nu, 
\end{equation}
where $Var(B^i)$ denotes the variation process of $B^i$, $i = 1,\dots,d$. Then $B$, $C$, and $\nu$ are absolutely continuous with respect to $A$, therefore 
$$B = b\cdot A,\quad C = c\cdot A,\quad {\rm and}\quad \nu = \eta\cdot A,$$
where $b$ is a predictable $\mathbb R^d$-valued process, $c$ is a predictable process with values in the set of nonnegative-definite matrices, and $\nu$ is a predictable Levy-measure-valued process.
Let us define a vector-valued process $R^{\{\widehat{\pi}\}}$ as

\begin{equation}\label{8121}
\begin{array}{rcl}
R^{\{\widehat{\pi}\}} &:=& R - (c{\widehat{\pi}})\cdot A - \left( \frac{\widehat {\pi}^{\top} x}{1+\widehat {\pi}^{\top}x }x\right)*\mu.\\
\end{array}
\end{equation} 
\textcolor{black}{Note that, the process $R^{\{\widehat{\pi}\}}$ governs the return of the traded assets under the numeraire $\frac{\widehat X(x,0)}{x}=\mathcal E\left({\widehat{\pi}}\cdot R\right)$.}
Here end below superscript $\top$ denotes the transpose of a vector.
Also note that $R^{\{\widehat{\pi}\}}$ is a semimartingale as 
$$\sum\limits_{i=0}^d\sum\limits_{s\leq \cdot}\left(\tfrac{{\widehat{\pi}}^{\top}_s \Delta R_s}{1 + {\widehat{\pi}}^{\top}_s\Delta R_s}\right)^2(\Delta \rho^i_s)^2<\infty.$$

Let $\mathcal M^\infty(x)$ denote the set of uniformly bounded elements of $\mathcal M^2(x)$. 
\begin{lem}\label{8111}
Let us assume that the assumptions of 
Theorem \ref{mainThm1} hold.
Then every element of $\mathcal M^\infty(x)$ be represented as a stochastic integral with respect to $R^{\{\widehat{\pi}\}}$.
\end{lem}
\begin{proof}
Let $M\in\mathcal M^\infty(x)$.
Then for a sufficiently large constant $C'>0$, we have

\begin{equation}\label{8281}
0<C' + M = C' + H\cdot S^X = 
\frac{C'\mathcal E\left(\tilde {\pi}\cdot R \right)}{\mathcal E\left( {\widehat{\pi}}\cdot R\right)},
\end{equation}
for some predictable and $R$-integrable process $\tilde \pi$.
First, as $\Delta (\widehat \pi \cdot R)>-1$, we have 
$$\frac{\mathcal E\left(\tilde \pi\cdot R \right)}{\mathcal E\left({\widehat{\pi}}\cdot R \right)} = \mathcal E(D),$$
where
$$D = \tilde \pi\cdot R - {\widehat{\pi}}\cdot R - [(\tilde \pi\cdot R)^c - ({\widehat{\pi}}\cdot R)^c, ({\widehat{\pi}}\cdot R)^c]-
\sum\limits_{t\leq \cdot}\left(\Delta (\tilde \pi\cdot R_t - {\widehat{\pi}}\cdot R_t)\frac{\Delta {\widehat{\pi}}\cdot R_t}{1 + \Delta {\widehat{\pi}}\cdot R_t}\right),$$
which is a (well-defined) semimartingale in view of finiteness of $\sum\limits_{t\leq \cdot} (\Delta \tilde \pi\cdot R_t)^2$ and $\sum\limits_{t\leq \cdot} (\Delta\widehat \pi\cdot R_t)^2$, see \cite[Lemma 3.4]{KarKar07}. 
Therefore, we can restate 
$\frac{\mathcal E\left(\tilde \pi\cdot R \right)}{\mathcal E\left({\widehat{\pi}}\cdot R \right)}$ as
\begin{equation}\label{772}
\frac{\mathcal E\left(\tilde \pi\cdot R \right)}{\mathcal E\left({\widehat{\pi}}\cdot R \right)} = \mathcal E\left((\tilde \pi - {\widehat{\pi}})\cdot R^{\{\widehat{\pi}\}}\right).
\end{equation}
Using representation \eqref{772}, in \eqref{8281} we obtain
$$C' + M = C'\mathcal E\left( (\tilde \pi - {\widehat{\pi}})\cdot R^{\{\widehat{\pi}\}}\right) = C'+ C'\left\{\mathcal E\left( (\tilde \pi - {\widehat{\pi}})\cdot R^{\{\widehat{\pi}\}}\right)_{-}(\tilde \pi - {\widehat{\pi}})\right\}\cdot R^{\{\widehat{\pi}\}}.$$
Solving for $M$, we get
$$M = \left\{C'\mathcal E\left( (\tilde \pi - {\widehat{\pi}})\cdot R^{\{\widehat{\pi}\}}\right)_{-}(\tilde \pi - {\widehat{\pi}})\right\}\cdot R^{\{\widehat{\pi}\}},$$
which completes the proof.
\end{proof}

Let $M^0$ and $M^1$ denote the solutions to \eqref{axxC} and \eqref{addC}, respectively. It follows from \cite[Lemma 6]{KS2006} that there exist sequences $(\bar M^{0,n})_{n\geq 1}$ and $(\bar M^{1,n})_{n\geq 1}$ in $\mathcal M^\infty(x)$, such that 
\begin{displaymath}
\lim\limits_{n\to\infty}\bar M^{0,n}_T = M^0_T\quad {\rm and}\quad \lim\limits_{n\to\infty}\bar M^{1,n}_T = M^1_T,\quad \Pas
\end{displaymath}
We suppose that  $\bar M^{0,n}$ is bounded by $n$, $n\geq 1,$ this without loss of generality. Therefore, the jumps of $\bar M^{0,n}$ are bounded by $2n$ and the quadratic variation of $\bar M^{0,n}$ is locally bounded, where 
$$T_k := \inf\left\{t\geq 0:~[\bar M^{0,n}]_t\geq k\right\},\quad k\geq 1,$$
is a localizing sequence for $[\bar M^{0,n}]$. Note that $[\bar M^{0,n}]_{T_k}\leq k + 4n^2$. Let us define
$$\tilde M^{0,n}_t := \bar M^{0,n}_{\min(t, T_n)},\quad t\in[0,T], n\geq 1.$$
Then $\tilde M^{0, n}$ is bounded by $n$, its quadratic variation is bounded $n+4n^2$, and its jumps are bounded by $2n$.   Moreover, by construction we have
$$\lim\limits_{n\to\infty}\tilde M^{0,n}_T = M^0_T,\quad \Pas$$
Analogously, we can construct a sequence $\tilde M^{1,n}$, $n\geq 1$, of martingales under $\mathbb R(x)$, such that $\tilde M^{1,n}$ is bounded by $n$, its quadratic variation is bounded by $n + 4n^2$, and its jumps are bounded by $2n$, $n\geq 1$, and such that 
  $$\lim\limits_{n\to\infty}\tilde M^{1,n}_T = M^1_T,\quad \Pas$$
 
Lemma \ref{8111} implies the existence of predictable $R^{\{\widehat{\pi}\}}$-integrable processes $\gamma^{0,n}$ and $\gamma^{1,n}$, $n\geq 1$, such that 
\begin{equation}\label{81111}
\gamma^{0,n}\cdot R^{\{\widehat{\pi}\}} = \frac{\tilde M^{0,n}}{x},\quad \gamma^{1,n}\cdot R^{\{\widehat{\pi}\}} = \frac{\tilde M^{1,n}}{x},\quad n\geq 1.
\end{equation}
We define the family of processes $(R^{\{\varepsilon\theta\}})_{\varepsilon\in(-\varepsilon_0, \varepsilon_0)}$ as 
\begin{equation}\label{Reps}
\begin{array}{rcl}
R^{\{\varepsilon\theta\}} &:=& R - \varepsilon(c\theta)\cdot A - \varepsilon\left(\frac{\theta^{\top}_sx}{1 +\varepsilon\theta^{\top}_sx}x\right)*\mu,\\
\end{array}
\end{equation}
\textcolor{black}{
where $R^{\{\varepsilon\theta\}}$ governs the returns of the traded assets under $N^\varepsilon$, and} similarly to the verification after \eqref{8121}, one can show that $R^{\{\varepsilon\theta\}}$ is a semimartingale for every $\varepsilon\in(-\varepsilon_0, \varepsilon_0)$.
Finally, let us define the family $\left(\tilde X^{\Delta x, \varepsilon, n}\right)_{(\Delta x, \varepsilon, n)\in(-x,\infty)\times(-\varepsilon_0,\varepsilon_0)\times \mathbb N}$ as
\begin{equation}\label{8122}
\tilde X^{\Delta x, \varepsilon, n} := (x + \Delta x) \mathcal E\left(\left( {\widehat{\pi}} + \Delta x\gamma^{0,n} + \varepsilon(-\theta + \gamma^{1,n}) \right)\cdot R^{\{\varepsilon\theta\}}\right).
\end{equation}

\begin{thm}\label{thmCorOptimizer}
Let $x>0$ be fixed and the assumptions of Theorem \ref{mainThm1} hold. Then we have.
\begin{enumerate}\item
For every $n\in\mathbb N$, there exists $\delta = \delta(n)>0$, such that, $$\tilde X^{\Delta x, \varepsilon, n}\in\mathcal X(x + \Delta x, \varepsilon),\quad (\Delta x, \varepsilon)\in B_{\delta(n)}(0,0),$$
where $B_{\delta(n)}(0,0)$ denotes a ball of radius ${\delta(n)}$ centered at $(0,0)$.
\item
There exists a function $n = n(\Delta x,\varepsilon):(-x,\infty)\times(-\varepsilon_0,\varepsilon_0)\to\mathbb N$, such that \begin{equation}\label{1121}
\mathbb E\left[U\left(\tilde X^{\Delta x, \varepsilon, n(\Delta x,\varepsilon)}_T\right)\right] = u(x + \Delta x, \varepsilon) - o(\Delta x^2 + \varepsilon^2).
\end{equation}
\item\textcolor{black}{The processes $\tilde X^{\Delta x, \varepsilon, n(\Delta x,\varepsilon)}$'s from the previous item have the following proportions invested in the corresponding stocks:
\begin{equation}\label{01134}
\left((1 - \varepsilon)I + \varepsilon \theta\vec 1^{\top}\right)\left(\widehat \pi + \Delta x\gamma^{0,n(\Delta x,\varepsilon)} + \varepsilon(-\theta + \gamma^{1,n(\Delta x,\varepsilon)}) \right),
\end{equation} 
where $I$ is $(d+1)\times(d+1)$ identity matrix and $\theta\vec 1^{\top}$ is the outer product of $\theta$ and the vector, whose every component equals to $1$.}
\end{enumerate}
\end{thm}

\begin{rem}\label{rem:DeltaXStrategies}
By taking $\varepsilon = 0$,
Theorem \ref{thmCorOptimizer} theorem gives corrections to optimal proportions invested in stock with respect to perturbations of the initial wealth only. In this case the nearly optimal family of wealth processes is given by 
 \begin{equation}\nonumber
\bar X^{\Delta x, n} := (x + \Delta x) \mathcal E\left(\left(\widehat \pi + \Delta x\gamma^{0,n}  \right)\cdot R\right),\quad (\Delta x, n)\in(-x,\infty)\times \mathbb N,
\end{equation}
where $\gamma^{0,n}$ are given in \eqref{81111}. Theorem \ref{thmCorOptimizer} asserts that there exists a function $n = n(\Delta x):(-x,\infty)\to\mathbb N$, such that 
\begin{displaymath}
\mathbb E\left[U\left(\bar X^{\Delta x, n(\Delta x)}_T\right)\right] = u(x + \Delta x, 0) - o(\Delta x^2).
\end{displaymath}
This allows to construct corrections to optimal trading strategies in the settings of \cite{KS2006}.
\end{rem}



\section{Relationship to the risk-tolerance wealth process}\label{secRiskTol}
We recall here that for an initial wealth $x>0$, the {\it risk-tolerance wealth process} is defined as a maximal wealth process $\mathcal R(x)$, such that 
\begin{equation}\label{1251riskTol}
\mathcal R_T(x) = -\frac{U'(\widehat X_T(x,0))}{U''(\widehat X_T(x,0))},
\end{equation}
i.e. it is a replication process for the random payoff given by the right-hand side of \eqref{1251riskTol}. The term {\it risk-tolerance wealth process} was introduced in \cite{KS2006b} in the context of asymptotic analysis of utility-based prices, in general it may not exist. 
 For $x>0$ and with $y = u_x(x,0)$, following \cite{KS2006b}, we change num\'eraser  in the $0$-model to $ \frac{\mathcal R(x)}{\mathcal R_0(x)}$, that is we set
\begin{equation}\nonumber
S^{\mathcal R(x)} := \left(\frac{\mathcal R_0(x)}{\mathcal  R(x)}, \frac{ \mathcal R_0(x)\mathcal E(\rho^1)}{\mathcal  R(x)}, \dots, \frac{ \mathcal R_0(x)\mathcal E(\rho^d)}{\mathcal  R(x)}\right).
\end{equation}
Next, we define 
\begin{equation}\nonumber
\frac{d\widetilde {\mathbb R}(x)}{d \mathbb P} := \frac{\mathcal R_T(x)}{\mathcal R_0(x)}\frac{ \widehat Y_T(y,0)}{y},
\end{equation}
and  \begin{equation}\nonumber
\mathcal {\widetilde M}^2(x,0) := \left\{ M\in \mathbf H_0^2(\mathbb{\widetilde R}(x)):~M = H\cdot S^{\mathcal R(x)}\right\},
\end{equation}
a space of square-integrable martingales under $\mathbb{\widetilde R}(x)$ starting from $0$, and denote by 
$\mathcal{\widetilde N}^2(y,0)$ the orthogonal complement of $\mathcal {\widetilde M}^2(x,0)$ in $\mathbf H_0^2(\mathbb{\widetilde R}(x))$. 
Theorem \ref{riskTolThm2} below relates the structural properties of the approximations in Theorems \ref{mainThm2}, \ref{162}, and \ref{161}  to a Galtchouk-Kunita-Watanabe decomposition (under the num\'eraire  $ \frac{\mathcal R(x)}{\mathcal R_0(x)}$ and measure $\mathbb{\widetilde R}(x)$), provided that $\mathcal R(x)$ exists. 
 Theorem \ref{riskTolThm2} is stated without a proof, as line by line adaptation of the proof of \cite[Theorem 8.3]{MostovyiSirbuModel} applies here.  
\begin{thm}\label{riskTolThm2}
Let us suppose that \eqref{finCond}, \eqref{NUPBR}, and Assumption \ref{rra} hold,  $x>0$ is fixed 
and the risk-tolerance process $\mathcal R(x)$ exists. 
Consider the (square-integrable) martingale 
$$P_t := \mathbb E^{\mathbb{\widetilde R}(x)}\left[xF\left(A(\widehat X_T(x,0)) -1\right) |\mathcal F_t\right],\ \ t\in[0,T],$$ 
and its  the Galtchouk-Kunita-Watanabe decomposition specified as 
\begin{equation}\label{eq:Kunita-Watanabe}
P=P_0-{\widetilde M}^1-{\widetilde N}^1,
\end{equation}
where $P_0\in \R$, ${\widetilde M}^1 \in \mathcal {\widetilde M}^2(x,0)$, and, for $y = u_x(x,0)$, ${\widetilde N}^1  \in \mathcal{\widetilde N}^2(y,0)$. 
Then, one can recover $M^1(x,0)$ and $N^1(y,0)$, the optimal solutions  to the auxiliary minimization problems  \eqref{addC} and \eqref{bddC}, through the Galtchouk-Kunita-Watanabe decomposition \eqref{eq:Kunita-Watanabe}
 as follows (by going back to the original num\'eraire):
\begin{equation}\nonumber
{\widetilde M}^1_t=
\frac{\widehat X_t(x,0) }{\mathcal R_t(x)}M_t^1(x,0),\quad {\widetilde N}^1_t=\frac{x}{y} N_t^1(y,0),\quad t\in[0,T].
\end{equation}
With
\begin{equation}\nonumber
C_a:= x^2\mathbb E^{\mathbb R(x)}\left[F^2 - G - \frac{F^2}{ A(\widehat X_T(x,0))}\right],
\end{equation}
 \begin{equation}\nonumber
 C_b:= y^2\mathbb E^{\mathbb R(x)}\left[G + F^2\left(1 - A\left(\widehat X_T(x,0)\right)\right) \right],
 \end{equation}
the components of the respective Hessian terms in the second-order expansion of $u$ and $v$, are given by 
\begin{equation}\nonumber
\begin{split}
a(\varepsilon, \varepsilon)
&=  \frac{\mathcal R_0(x)}{x}P_0^2+ \frac{\mathcal R_0(x)}{x}\mathbb E^{\mathbb {\widetilde R}(x)}\left[\left({\widetilde N}^1_T \right)^2 \right] +C_a\\
&= \frac{\mathcal R_0(x)}{x}\inf\limits_{\widetilde M\in\mathcal {\widetilde M}^2(x,0)}\mathbb E^{\mathbb {\widetilde R}(x)}\left[\left(\widetilde M_T+xF\left( - 1 +{A\left(\widehat X_T(x,0)\right)}\right) \right)^2 \right]  + 
C_a,
\end{split}
\end{equation}
and \begin{equation}\nonumber
\begin{split}
b(\varepsilon, \varepsilon)
&= \frac{\mathcal R_0(x)}{x}\left (\frac yx  \right)^2P_0^2+\frac{\mathcal R_0(x,0)}{x}\left (\frac yx  \right)^2 \mathbb E^{\mathbb{\widetilde R}(x)}\left[ \left({\widetilde M}^1_T \right)^2\right] + C_b\\
&= \frac{\mathcal R_0(x)}{x}\inf\limits_{ {\widetilde N}\in\mathcal{ N}^2(y,0)}\mathbb E^{\mathbb{\widetilde R}(x)}\left[ \left({\widetilde N}_T + yF\left(  - 1 + A\left(\widehat X_T(x,0)\right)\right)\right)^2\right] + C_b.
\end{split}
\end{equation}
We also have 
$$a(x, \varepsilon)=P_0\quad  { and} \quad b(y, \varepsilon)=  \frac{yP_0}{x a(x,x)}.$$
The  conclusions  of Theorem \ref{mainThm2}, with these notations, hold true.
\end{thm}

\begin{rem} In many references, in order to call \eqref{eq:Kunita-Watanabe} the Kunita-Watanabe decomposition of $P$, one additionally needs $\widetilde N^1$ to be orthogonal to $S^{\mathcal R(x)}$, which amounts to $\widetilde N^1S^{\mathcal R(x)}$ being a martingale under ${\mathbb{\widetilde R}(x)}$. Some authors, see e.g., \cite[p. 2181]{KS2006b}, do not require~this.
\end{rem}


\section{Proofs}\label{secProofs} \subsection{Characterization of primal and dual admissible sets}
An important characterization of the primal and dual admissible sets after perturbations is contained in the following lemma.

\begin{lem}\label{keyConcreteCharLem}
Under Assumption \eqref{NUPBR}, 
for every $\varepsilon \in(-\varepsilon_0, \varepsilon_0)$, we have
\begin{equation}\label{01133}
\mathcal X(1,\varepsilon) = \mathcal X(1,0)\frac{1}{N^\varepsilon},
\end{equation}
\begin{equation}\label{01132}
\mathcal Y(1,\varepsilon) = \mathcal Y(1,0)N^\varepsilon,
\end{equation}
where we have used the following notations
\begin{displaymath}
\begin{split}
 \mathcal X(1,0)\frac{1}{N^\varepsilon}&=\left\{\frac{X}{N^\varepsilon} = \left(\frac{X_t}{N^\varepsilon_t}\right)_{t\in[0,T]}:~X\in \mathcal X(1,0)\right\},\\
\mathcal Y(1,0)N^\varepsilon &=\left\{{Y}{N^\varepsilon} = \left({Y_t}{N^\varepsilon_t}\right)_{t\in[0,T]}:~Y\in \mathcal Y(1,0)\right\}.
\end{split}
\end{displaymath}
In particular, both $\mathcal X(1,\varepsilon)$ and $\mathcal Y(1,\varepsilon)$ are non-empty and no unbounded profit with bounded risk holds for every $\varepsilon \in(-\varepsilon_0, \varepsilon_0)$.
\end{lem}
\begin{proof}
Let us fix $\varepsilon \in(-\varepsilon_0, \varepsilon_0)$. Then, for an arbitrary predictable and $S^\varepsilon$-integrable process $\psi$, such that $\Delta (\psi\cdot S^\varepsilon)>-1$, let us set $X^\varepsilon := \mathcal E\left(\psi\cdot S^\varepsilon\right)$. Then $X^\varepsilon \in\mathcal X(1, \varepsilon)$. Let us consider $X^0 := X^\varepsilon \mathcal E\left(-\varepsilon\bar R\right)$. One can see that $X^0\in\mathcal X(1,0)$. This implies that $$\mathcal X(1,\varepsilon)N^\varepsilon\subseteq  \mathcal X(1,0).$$ Similarly, one can show the reverse inclusion. Therefore,  \eqref{01133} is valid. 

Let us fix $Y\in \mathcal Y(1,0)$ and take an arbitrary $\tilde X^\varepsilon\in\mathcal X(1,\varepsilon)$. By \eqref{01133}, $\tilde X^\varepsilon N^\varepsilon\in\mathcal X(1,0)$. Therefore, $Y\tilde X^\varepsilon N^\varepsilon$ is a supermartingale. We deduce that $YN^\varepsilon\in \mathcal Y(1,\varepsilon).$ As a consequence, we have 
$$\mathcal Y(1,0)N^\varepsilon\subseteq \mathcal Y(1,\varepsilon).$$ 
In a similar manner, one can show that $\mathcal Y(1,0)N^\varepsilon\supseteq \mathcal Y(1,\varepsilon)$. As a result, \eqref{01132} holds.
\end{proof}
We will need the following lemma from \cite{MostovyiSirbuModel}.
\begin{lem}[Mostovyi, Sirbu, 2017]\label{rraCor}
Under Assumption \ref{rra}, for every $z>0$ and $x>0$, we have
\begin{displaymath}\begin{split}
U'(zx) &\leq \max\left(z^{-c_2},1\right)U'(x)\leq \left(z^{-c_2}+1\right)U'(x),\\
-V'(zx) &\leq \max\left(z^{-\frac{1}{c_1}},1\right)(-V'(x))\leq \left(z^{-\frac{1}{c_1}}+1\right)(-V'(x)).
\end{split}\end{displaymath}
\end{lem}

For brevity of notations in the proof of Lemma \ref{12132} below, we denote by $G^c$ the continuous part of $[\bar R,\bar R]$ evaluated at $T$ and let $H_i$, where $H_i$ takes values in $\left[-\tfrac{1}{2\varepsilon_0},\frac{1}{2\varepsilon_0}\right]$,  $i\in\mathbb N$, are the jumps of $\bar R$ up to $T$. 
Note that, with $G$ being defined in \eqref{defFG}, we have
\begin{equation}\label{841}
G^c + \sum\limits_{i= 1}^{\infty}H_i^2 = G,\quad \Pas
\end{equation}
We define
\begin{equation}\nonumber
{\tilde N}^{\varepsilon}:= \exp\left( -\varepsilon F - \tfrac{1}{2}\varepsilon^2 G^c + \sum\limits_{i=1}^{\infty}\left(\log(1 - \varepsilon H_i) + \varepsilon H_i\right)\right),\quad \varepsilon \in(-\varepsilon_0,\varepsilon_0),
\end{equation}
and observe that the series $\sum\limits_{i=1}^{\infty}\left(\log(1 - \varepsilon H_i) + \varepsilon H_i\right)$ converges absolutely  for every $\varepsilon \in(-\varepsilon_0,\varepsilon_0)$, $\Pas$, in view of \eqref{841} and since $|\log(1 + x) - x| \leq x^2$ for every $x\in\left[-\tfrac{1}{2},\tfrac{1}{2}\right]$. 
\begin{lem}\label{12132}
Let $x>0$ be fixed and the conditions of Theorem \ref{mainThm1} hold, and $y=u_x(x,0)$.
Let $\alpha^0$ and $\alpha^1$ be the terminal values of some elements of $\mathcal M^\infty(x)$. With $\xi := \widehat X_T(x,0)$ denoting the solution to \eqref{primalProblem} corresponding to $x>0$ and $\varepsilon = 0$, we define 
\begin{equation}\nonumber
\begin{split}
\psi(s,t) &:= \frac{1}{x}\left(x + s(1 + \alpha^0) + t\alpha^1\right)\frac{1}{\tilde N^t},\\
w(s,t) &:= \mathbb E\left[U(\xi \psi(s,t))\right],\quad (s,t)\in\mathbb R\times(-\varepsilon_0,\varepsilon_0).
\end{split}
\end{equation}
Then $w$ possesses the second-order expansion at $(0,0)$, given by
\begin{equation}\nonumber
w(s,t) = w(0,0) +(s\quad t) \nabla w(0,0) + \tfrac{1}{2}(s\quad t) 
H_w
\begin{pmatrix}
s\\
t\\
\end{pmatrix} + o( t^2 +s^2);
\end{equation}
here the components of the gradient are given by
\begin{displaymath}
\begin{split}
w_s(0,0)&=u_x(x,0)\quad and \quad 
w_t(0,0)=xy\mathbb E^{\mathbb R(x)}\left[ F\right],
\end{split}
\end{displaymath}
and where the Hessian is of the form
\begin{displaymath}
H_w := \begin{pmatrix} 
  w_{ss}(0,0)     & w_{st}(0,0)\\ 
  w_{st}(0,0) 	& w_{tt}(0,0)\\ 
\end{pmatrix},
\end{displaymath}
 where  
 \begin{displaymath}\begin{split}
 w_{ss}(0,0) & = -\frac{y}{x}\mathbb E^{\mathbb R(x)}\left[ A(\xi)(1 + \alpha^0)^2\right],
 \\
 w_{st}(0,0) & = -\frac{y}{x}\mathbb E^{\mathbb R(x)}\left[A(\xi)(xF + \alpha^1)(1 + \alpha^0) - xF(1 + \alpha^0) \right],\\
 w_{tt}(0,0) & = -\frac{y}{x}\mathbb E^{\mathbb R(x)}\left[A(\xi)(\alpha^1 + xF)^2 - 
2xF\alpha^1 - x^2(F^2 + G) \right],
 \end{split}\end{displaymath}
 are the second-order partial derivatives of $w$ at $(0,0)$.
\end{lem}
\begin{proof}
From boundedness of  $\alpha^0$ and $\alpha^1$, it follows that there exists a constant $\varepsilon\in(0,\min(\varepsilon_0, 1))$, such that
\begin{equation}\label{11510}
\varepsilon\left(|\alpha^0+ 1| + |\alpha^1|\right)\leq \frac{x}{2}.
\end{equation}
Let us fix an arbitrary $(s,t)\in B_{\varepsilon}(0,0)$ and define
\begin{displaymath}
\widetilde \psi (z) := \psi(zs, zt),\quad z\in (-1,1).
\end{displaymath}
As by construction of $(H_k)_{k\in\mathbb N}$, see \eqref{841}, we have that $\sum\limits_{k\geq 1}\left(\log(1 - t H_k )+t H_k\right)$ converges for every $t \in [-\varepsilon/2,\varepsilon/2]$, $\Pas$, and the series of term by term derivatives, $\sum\limits_{k\geq 1}\frac{-t H_k^2}{1 - t H_k}$, converges uniformly in $t \in [-\varepsilon/2,\varepsilon/2]$, where $\frac{t H_k^2}{1 - t H_k}$ is continuous in $t$ on $[-\varepsilon/2,\varepsilon/2]$ for every $k\geq 1$, we deduce that
$$ -\frac{\partial}{\partial t} \sum\limits_{k\geq 1}\left(\log(1 - t H_k )+t H_k\right) =  t\sum\limits_{k\geq 1}\frac{H_k^2}{1 - t H_k},\quad t \in(-\varepsilon/2,\varepsilon/2),$$
and we get
$$\psi_t(s, t) = \frac{\alpha^1}{x\tilde N^t} + \psi(s,t)\left(F + t G^c + t \sum\limits_{k\geq 1} \frac{H_k^2}{1 - t H_k}\right)\quad {\rm and} \quad 
\psi_s(s, t) = \frac{1 + \alpha^0}{x\tilde N^t},$$
Consequently, we obtain
\begin{equation}\label{1155}
\begin{split}
\widetilde \psi'(z) &= \psi_s(sz, tz)s + \psi_t(sz, tz)t \\
&=\frac{1 + \alpha^0}{x\tilde N^{zt}}s + \left(\frac{\alpha^1}{x\tilde N^{zt}} + \widetilde \psi(z) \left(F + zt G^c+ zt \sum\limits_{k\geq 1} \frac{H_{k}^2}{1 - zt H_{k}}\right)\right)t.
\end{split}
\end{equation}
Similarly, since $\Pas$, $\sum\limits_{k\geq 1} \frac{t  H_k^2}{1 - t H_k}$ converges for every $t\in[-\varepsilon/2,\varepsilon/2]$,  since the series of term by term partial derivatives, $\sum\limits_{k\geq 1} \frac{H_k^2}{(1 - t H_k)^2}$, converges uniformly in $t\in [-\varepsilon/2,\varepsilon/2]$, and from continuity of $\frac{H_k^2}{(1 - t H_k)^2}$ in $t$ on $[-\varepsilon/2,\varepsilon/2]$ for every $k\geq 1$, we deduce that
$$\frac{\partial}{\partial t}\left(\sum\limits_{k\geq 1} \frac{t  H_k^2}{1 - t H_k}\right) = \sum\limits_{k\geq 1} \frac{H_k^2}{\left(1 - t H_k\right)^2},
\quad t \in (-\varepsilon/2,\varepsilon/2),$$
and we get
\begin{equation}\nonumber\begin{split}
\psi_{tt}(s, t) &= \frac{2\alpha^1}{xN^{t}}\left(F + t G^c+ t \sum\limits_{k\geq 1} \frac{H_k^2}{1 - t H_k}\right) \\
&\quad+ \psi(s,t) \left(\left(F + t G^c + t \sum\limits_{k\geq 1} \frac{H_k^2}{1 - t H_k}\right)^2 + G^c+\sum\limits_{k\geq 1}\frac{H_k^2}{(1 - t H_k)^2} \right),\\
\psi_{st}(s, t) &= \frac{1 + \alpha^0}{x\tilde N^{t}}\left(F + t G^c + t \sum\limits_{k\geq 1} \frac{H_{k}^2}{1 - t H_{k}}\right),
\end{split}\end{equation}
and also $\psi_{ss}(s, t)$ is identically equal to  $0$. 
Therefore, we get
\begin{displaymath}
\begin{split}
\widetilde \psi''(z) &= t^2\psi_{tt}(zs, zt) + 2ts\psi_{st}(zs, zt) + s^2\psi_{ss}(zs, zt)\\
 &= \left(\frac{2\alpha^1}{xN^{zt}}\left(F + zt G^c+ zt \sum\limits_{k\geq 1} \frac{H_k^2}{1 - zt H_k}\right) \right.\\
 &\quad\left.+ \tilde\psi(z) \left(\left(F + zt G^c + zt \sum\limits_{k\geq 1} \frac{H_k^2}{1 - zt H_k}\right)^2 + G^c+\sum\limits_{k\geq 1}\frac{H_k^2}{(1 - zt H_k)^2}\right) \right)t^2 \\
 &\quad+ 2 \frac{1 + \alpha^0}{x\tilde N^{zt}}\left(F + zt G^c + zt \sum\limits_{k\geq 1} \frac{H_{k}^2}{1 - zt H_{k}}\right)ts.
\end{split}
\end{displaymath}
Setting $W(z) := U(\widetilde \psi(z)\xi )$, where $z\in(-1, 1)$, by pointwise differentiation, we get
\begin{equation}\label{1152}\begin{split}
W'(z) &= \xi \widetilde \psi'(z)U'(\xi \widetilde \psi(z)), \\
W''(z) &= \left(\xi \widetilde \psi'(z)\right)^2U''(\xi \widetilde \psi(z)) + \xi \widetilde \psi''(z)U'(\xi \widetilde \psi(z)).
\end{split}
\end{equation}
Let us define 
$$\quad J := 1+|F| + G.$$
As $$\left|zt G^c+ zt \sum\limits_{k\geq 1} \frac{H_{k}^2}{1 - zt H_{k}}\right|\leq 2|zt| G, \quad z\in(-1,1),$$
from \eqref{1155} using \eqref{11510} and since 
\begin{equation}\nonumber
\frac{1}{2}\leq \widetilde \psi(z)\tilde  N^{zt}\leq \frac{3}{2}, \quad z\in (-1,1),
\end{equation}  
One can see that  there exists a constant $b_1>0$, for which we have  
 $$|\widetilde \psi'(z) |\leq b_1J\exp(b_1\varepsilon J),\quad {\rm and}\quad  \widetilde \psi(z) ^{-c_2} + 1\leq b_1 \exp(b_1 \varepsilon J),\quad z\in(-1, 1).$$
Therefore, from \eqref{1152} using Lemma \ref{rraCor}, we obtain
\begin{equation}\label{1133}
\begin{split}
b_1^2  U'(\xi)\xi J\exp(2b_1\varepsilon J) \geq
\sup\limits_{z\in (-1, 1)}U'(\xi)\xi\left( (\widetilde \psi(z))^{-c_2}+1\right)\left|\widetilde \psi'(z)\right| \geq 
\sup\limits_{z\in (-1, 1)}|W'(z) |.
\end{split}
\end{equation}
Similarly, from \eqref{1152} applying  Assumption \ref{rra} and Lemma \ref{rraCor}, one can show that there exists a constant $b_2 > 0$, for which we have 
\begin{equation}\label{1141}
b_2 U'(\xi)\xi J^2\exp(b_2\varepsilon J)\geq \sup\limits_{z\in (-1, 1)}|W''(z) |.
\end{equation}
Assembling \eqref{1133} and \eqref{1141}, we conclude that 
\begin{equation}\nonumber
U'(\xi)\xi\left(
b_1^2  J\exp(2b_1\varepsilon J) +
b_2  J^2\exp(b_2\varepsilon J)\right)
\geq
\sup\limits_{z\in (-1, 1)}\left(|W'(z) | + |W''(z) |\right).
\end{equation}
Consequently, as $1\leq J\leq J^2$, one can find a constant $b>0$ such that for every $z_1$ and $z_2$ in $(-1, 1)$, we get
\begin{equation}\label{1157} 
b U'(\xi)\xi
  J^2\exp(b\varepsilon J)\left|{z_1 -z_2}\right|\geq \left|{W(z_1) - W(z_2)}\right| + \left|{W'(z_1) - W'(z_2)}\right|  
.
\end{equation}
H\"older's inequality (possibly for a smaller $\varepsilon$) together with Assumption \ref{integrabilityAssumption} assert that the right-hand side of \eqref{1157} integrable. Since the bound in \eqref{1157} is uniform in $(s,t)\in B_{\varepsilon}(0,0)$, applying the dominated convergence theorem we deduce the assertions of the lemma.
\end{proof}
\subsection{Proofs of Theorems \ref{mainThm1}, \ref{mainThm2}, \ref{162}, and \ref{161}}\label{secProofC}


Let us consider the closures in $\mathbb L^0(\mathbb P)$ of the convex solid hulls of the terminal values of elements of the primal and dual domains for the $0$-models, that is of $\left\{X_T:~X\in\mathcal X(1,0)\right\}$ as well as $\left\{Y_T:~Y\in\mathcal Y(1,0)\right\}$. From \eqref{NUPBR}, it follows that they  satisfy \cite[Assumption 5.1]{MostovyiSirbuModel}. Using Lemma \ref{keyConcreteCharLem}, we get
\begin{displaymath}
\begin{split}
\left\{\frac{X_T}{N^\varepsilon_T}:~X\in\mathcal X(1,0)\right\} &= \left\{{X_T}:~X\in\mathcal X(1,\varepsilon)\right\},\\
\left\{Y_TN^\varepsilon_T:~Y\in\mathcal Y(1,0)\right\} &= \left\{Y_T:~Y\in\mathcal Y(1,\varepsilon)\right\},\quad \varepsilon\in(-\varepsilon_0, \varepsilon_0).
\end{split}
\end{displaymath}
 Consequently, the respective closures of convex solid hulls of
$$\left\{{X_T}:~X\in\mathcal X(1,\varepsilon)\right\}\quad
  {\rm and} \quad
  \left\{Y_T:~Y\in\mathcal Y(1,\varepsilon)\right\}$$
 also satisfy \cite[Assumption 5.1]{MostovyiSirbuModel} for every $\varepsilon\in(-\varepsilon_0, \varepsilon_0)$.   
  
   From Assumption \ref{sigmaBoundedness} and \cite[Lemma 6]{KS2006}, we deduce that the sets $\mathcal M^2(x)$ and $\mathcal N^2(x)$ satisfy \cite[Assumption 5.3]{MostovyiSirbuModel}. With the notations \eqref{defFG}, using Assumption \ref{boundedJumps}, we get 
\begin{equation}\nonumber
\max\left(N^\varepsilon_T, \frac{1}{N^\varepsilon_T}\right) \leq \exp\left(|\varepsilon F| + \varepsilon^2G\right), \quad \varepsilon\in(-\varepsilon_0, \varepsilon_0).
\end{equation}
Therefore, Assumption \ref{integrabilityAssumption} is analogous to \cite[Assumption 5.2]{MostovyiSirbuModel}. 

  In view of Lemma \ref{12132}, from which the greatest lower bound for the quadratic expansion of $u$ can be obtained, the least upper bound for $v$ can be obtained similarly.  Moreover, even though in \cite{MostovyiSirbuModel} and the present paper, the perturbations are different, the second-order expansions for the value functions, which stem from Lemma \ref{12132} and its consequences, coincide (here and in \cite{MostovyiSirbuModel}). Now, in view of the structures of perturbations represented by $N^\varepsilon_T$ here and by $L^\delta$ in \cite[p.14]{MostovyiSirbuModel},
   the assertions of Theorems \ref{mainThm1}, \ref{mainThm2}, \ref{162}, and \ref{161} follow from 
the line by line adaptation of the proofs of \cite[Theorem 5.4, Theorem 5.6, Theorem 5.7, and Theorem 5.8]{MostovyiSirbuModel}, respectively. Further details are not included for the brevity of the exposition.

\subsection{Proofs of the assertions from section \ref{approxTradingStrategies}}
In order to prove Theorem \ref{thmCorOptimizer}, first,  the following technical lemma has to be established. For $(\Delta x,\varepsilon,n)\in (-x,\infty)\times (-\varepsilon_0,\varepsilon_0)\times \mathbb N$, let us recall that 
$\nabla u(x,0)$, $H_u(x,0)$, and  $\tilde X^{\Delta x, \varepsilon, n}$'s are defined in  \eqref{gradientC}, \eqref{12136C}, and \eqref{8122}, respectively, and set 
\begin{equation}\label{5191}
f(\Delta x, \varepsilon,n):= \frac{u(x,0) + (\Delta x\quad \varepsilon) \nabla u(x,0) + \tfrac{1}{2}(\Delta x\quad \varepsilon) 
H_u(x,0)
\begin{pmatrix}
\Delta x\\
\varepsilon\\
\end{pmatrix} - \mathbb E\left[U\left(\tilde X^{\Delta x, \varepsilon, n}_T\right) \right]}{\Delta x^2 + \varepsilon^2}.
\end{equation}

\begin{lem}\label{51913}
Let us fix $x>0$ and suppose that the validity of the assumptions of Theorem~\ref{mainThm1}.  Let us consider $f$ defined in \eqref{5191}. Then, there exists a  function $g$, which is monotone and such that 
\begin{equation}\label{51911}
g(n)\geq\lim\limits_{|\Delta x|+|\varepsilon|\to 0}f(\Delta x, \varepsilon,n),\quad n\in\mathbb N,
\end{equation}
as well as \begin{equation}\label{51912}
\lim\limits_{n \to \infty}g(n) = 0.
\end{equation}

\end{lem}
\begin{proof} Using the argument of Lemma \ref{12132}, we essentially get the assertions of the lemma. Therefore, we only present the key steps. 
For a fixed $\varepsilon >0$, let us define 
\begin{equation}\nonumber
\begin{split}
\psi(\Delta x,\varepsilon) &:= \frac{x+\Delta x}{x}\frac{\mathcal E\left((\Delta x\gamma^{0,n} + \varepsilon \gamma^{1,n})\cdot R^{\{\widehat{\pi}\}}\right)}{ \mathcal E\left((\varepsilon\theta) \cdot R\right)},\\
w(\Delta x,\varepsilon) &:= \mathbb E\left[U(\widehat X_T(x,0)\psi(\Delta x,\delta))\right],\quad (\Delta x,\varepsilon)\in (-x,\infty)\times(-\varepsilon_0, \varepsilon_0),\\
\end{split}
\end{equation}
Let us choose $\varepsilon' > 0$, 
then let us pick $(\Delta x,\varepsilon)\in B_{\varepsilon'}(0,0)$, and then set
\begin{displaymath}
\widetilde \psi (z) := \psi(z\Delta x, z\delta),\quad z\in (-1,1).
\end{displaymath}
Setting $W(z) := U(\widetilde \psi(z)\widehat X_T(x,0))$, where $z\in(-1, 1)$, by pointwise differentiation, we obtain
\begin{equation}\nonumber
\begin{split}
W'(z) &= U'( \widetilde \psi(z)\widehat X_T(x,0)) \widetilde \psi'(z)\widehat X_T(x,0), \\
W''(z) &= U''(\widetilde \psi(z)\widehat X_T(x,0))\left( \widetilde \psi'(z)\widehat X_T(x,0)\right)^2 + U'(\widetilde \psi(z)\widehat X_T(x,0))\widetilde \psi''(z)\widehat X_T(x,0).
\end{split}
\end{equation}
Following the argument in Lemma \ref{12132}, boundedness of $\gamma^{0,n}\cdot R^{\{\widehat{\pi}\}} = \tilde M^{0,n}$, $\gamma^{1,n}\cdot R^{\{\widehat{\pi}\}}= \tilde M^{1,n}$, their quadratic variations and jumps, via Assumption \ref{integrabilityAssumption} and Lemma \ref{rraCor}, implies that for some random variable $\eta$ depending on $\varepsilon'$ and which is in $\mathbb L^1(\mathbb P)$ for a sufficiently small $\varepsilon'$, we have
\begin{equation}\nonumber
\left|{W(z_1) - W(z_2)}\right| + \left|{W'(z_1) - W'(z_2)}\right| \leq \eta\left|{z_1 -z_2}\right|.
\end{equation}

The  derivatives of $W$ plugged inside the expectation result in the exact form of the gradient $\nabla u(x,0)$  and the 
Hessian $H^{n}_u(x,0)$, such that $\lim\limits_{n\to\infty}H^{n} _u(x,0) = H_u(x,0)$. This results in the existence of a function $g$ satisfying \eqref{51911} and  \eqref{51912}. Finally, $g$  can be selected to be monotone.
\end{proof}
\begin{proof}[Proof of Theorem \ref{thmCorOptimizer}]
Let us fix $n\in\mathbb N$ and  consider $$(\gamma^{0, n} + \gamma^{1,n})\cdot R^{\{\widehat{\pi}\}} = \tilde M^{0,n} + \tilde M^{1,n}\in\mathcal M^\infty(x).$$ By construction, the jumps of this process process are bounded by $4n$. Therefore, setting $\delta(n):= \min\left(\varepsilon_0, \tfrac{1}{9n}\right)$, we obtain that for every $(\Delta x,\varepsilon)\in B_{\delta(n)}(0,0)$, the jumps of 
\begin{equation}\nonumber
\Delta x\tilde M^{0,n} +\varepsilon \tilde M^{1,n} \quad{\rm and}\quad 
(\varepsilon\theta)\cdot R
\end{equation}
take values in $(-1,1)$. 
Consequently, for every $(\Delta x,\varepsilon)\in B_{\delta(n)}(0,0)$, we get 
$$\mathcal E\left((\Delta x\gamma^{0,n} + \varepsilon \gamma^{1,n})\cdot R^{\{\widehat{\pi}\}}\right)>0 \quad {\rm and}\quad 
 \mathcal E\left((\varepsilon\theta) \cdot R\right)>0.$$
Therefore, via direct computations, we obtain
 \begin{displaymath}
0<\mathcal E\left(\widehat \pi \cdot R\right)\frac{\mathcal E\left((\Delta x\gamma^{0,n} + \varepsilon \gamma^{1,n})\cdot R^{\{\widehat{\pi}\}}\right)}{ \mathcal E\left((\varepsilon\theta) \cdot R\right)}
 =\frac{\mathcal E\left((\widehat \pi+\Delta x\gamma^{0,n} + \varepsilon \gamma^{1,n}) \cdot R\right)} 
 { \mathcal E\left((\varepsilon\theta) \cdot R\right)}
 = \frac{\tilde X^{\Delta x, \varepsilon, n}}{x+\Delta x}.
 \end{displaymath}
 In view of Lemma \ref{keyConcreteCharLem}, this implies that \begin{equation}\label{1141}
 \tilde X^{\Delta x, \varepsilon, n}\in\mathcal X(x+\Delta x, \varepsilon),\quad (\Delta x,\varepsilon)\in B_{\delta(n)}(0,0).\end{equation} This completes the proof of the first assertion of the theorem. 
 
 In order to prove the second assertion, we proceed as follows. 
Using Lemma \ref{51913}, we assert the existence of a monotone function $g$ satisfying \eqref{51911} and \eqref{51912} for $f$ defined in \eqref{5191}.
We set
$$\Phi(n) := \left\{(\Delta x,\varepsilon):~f(t\Delta x, t\varepsilon, n)\leq 2 g(n),~~{\rm for~every}~t\in[0,1]\right\},\quad n\in\mathbb N,$$
$$m(n) := {2}\inf\left\{m\geq n:~B_{1/m}(0,0)\subseteq \Phi(n)\right\}, \quad n\in\mathbb N.$$
Note that $m(n)<\infty$ for every $n\in\mathbb N$. With
$$n(\Delta x,\varepsilon) := \min\left\{ n\in\mathbb N:~ {m(n)} \geq \frac{1}{\sqrt{\Delta x^2 + \varepsilon^2}}\right\}, \quad (\Delta x,\varepsilon)\in(-x,\infty)\times(-\varepsilon_0,\varepsilon_0),$$
we have
\begin{equation}\label{1142}
\lim\limits_{|\Delta x|+|\varepsilon|\to 0}\frac{u(x + \Delta x, \varepsilon) - \mathbb E\left[U\left(\tilde X^{\Delta x, \varepsilon, n( \Delta x,\varepsilon)}_T\right) \right]}{\Delta x^2 + \varepsilon^2} = 0.\end{equation}
 
%
\textcolor{black}{In order to prove the third assertion of this theorem,  
let us consider
$$S^\varepsilon = \left(\frac{1}{N^\varepsilon},\frac{\mathcal E(\rho^1)}{N^\varepsilon},\dots,\frac{\mathcal E(\rho^d)}{N^\varepsilon}  \right),$$
the $(d+1)$-dimensional stock price process under $N^\varepsilon$. By direct computations, we get
\begin{equation}\label{9221}\begin{array}{rcl}
\left(\frac{1}{N^\varepsilon},\frac{\mathcal E(\rho^1)}{N^\varepsilon},\dots,\frac{\mathcal E(\rho^d)}{N^\varepsilon}  \right)
&=&
\left(\frac{1}{\mathcal E\left( (\varepsilon \theta)\cdot R\right)},\frac{\mathcal E(\rho^1)}{\mathcal E\left( (\varepsilon \theta)\cdot R\right)},\dots,\frac{\mathcal E(\rho^d)}{\mathcal E\left( (\varepsilon \theta)\cdot R\right)}  \right)\\
&=&
\left(\mathcal E\left( (e^0-\varepsilon \theta)\cdot R^{\{\varepsilon\theta\}}\right),\dots,\mathcal E\left( (e^d-\varepsilon \theta)\cdot R^{\{\varepsilon\theta\}}\right)\right),\\
\end{array}
\end{equation}
where $e^i$ is the constant-valued process whose $i$-th component equals to $1$ and all other components equal to zero at all times and $\left(R^{\{\varepsilon\theta\}}\right)_{\varepsilon\in(-\varepsilon_0, \varepsilon_0)}$ defined in \eqref{Reps}. 
Therefore, introducing the vector of returns under the num\'eraire $N^\varepsilon$, $R^\varepsilon$, from \eqref{9221} we get
$$R^\varepsilon =\frac{1}{1-\varepsilon} \left((e^0-\varepsilon \theta)\cdot R^{\{\varepsilon\theta\}},\dots,(e^d-\varepsilon \theta)\cdot R^{\{\varepsilon\theta\}}\right),$$
equivalently 
\begin{equation}\label{1251}
R^\varepsilon = \frac{1}{1-\varepsilon} (I - \varepsilon \vec 1 \theta^{\top})\cdot R^{\{\varepsilon\theta\}},
\end{equation}
where $\frac{1}{1-\varepsilon}$ is a normalization constant.} 

%
\textcolor{black}{Following the construction above, see \eqref{1141} and \eqref{1142}, for every $(\Delta x,\varepsilon)$ in a certain neighborhood of the origin, one can find $n(\Delta x,\varepsilon)$, such that $\tilde X^{\Delta x, \varepsilon, n(\Delta x,\varepsilon)}$'s form a family of  wealth processes that match the indirect utility up to the second order. To show that the corrections to optimal proportions (invested in the corresponding stocks) are given by \eqref{01134}, 
 for every $\varepsilon$ being sufficiently close to $0$ and every $\Delta x>-x$, we need to show that 
 $\tilde X^{\Delta x, \varepsilon, n}$'s defined in \eqref{8122} can be represented as
\begin{equation}\label{8172}
\tilde X^{\Delta x, \varepsilon, n} = (x + \Delta x) \mathcal E\left(\left(\left(\widehat \pi + \Delta x\gamma^{0,n} + \varepsilon(-\theta + \gamma^{1,n}) \right)^\top\left((1 - \varepsilon)I + \varepsilon\vec 1 \theta^{\top}\right)\right)^{\top}\cdot R^\varepsilon\right).
\end{equation}
 Here $ \mathcal E\left(\left(\left(\widehat \pi + \Delta x\gamma^{0,n} + \varepsilon(-\theta + \gamma^{1,n}) \right)^\top\left((1 - \varepsilon)I + \varepsilon\vec 1 \theta^{\top}\right)\right)^{\top}\cdot R^\varepsilon\right)\in\mathcal X(1, \varepsilon)$,  by the subsequent argument. We recall that $\theta^0_t = 1 - \sum\limits_{i=1}^d \theta^i_t$, $t\in[0,T]$, as the $\mathcal E(\theta\cdot R)$ is a wealth process of a self-financing portfolio, 
 and therefore,$$\vec 1^{\top} \theta  \equiv 1.$$
Consequently, we have
\begin{equation}\label{8241}
\begin{split}
&\mathcal E\left(\left(\left(\widehat \pi + \Delta x\gamma^{0,n} + \varepsilon(-\theta + \gamma^{1,n}) \right)^\top\left((1 - \varepsilon)I + \varepsilon\vec 1 \theta^{\top}\right)\right)^{\top}\cdot R^\varepsilon\right)\\=& 
\mathcal E\left(\left(\left(\widehat \pi + \Delta x\gamma^{0,n} + \varepsilon(-\theta + \gamma^{1,n}) \right)^\top\left(I + \frac{\varepsilon \vec 1 \theta^{\top}}{1 - \varepsilon \vec 1^{\top} \theta}\right)\right)^\top\cdot ((1 - \varepsilon)R^\varepsilon)\right) \\
= &
\mathcal E\left(\left(\widehat \pi + \Delta x\gamma^{0,n} + \varepsilon(-\theta + \gamma^{1,n}) \right)\cdot R^{\{\varepsilon\theta\}}\right) \\
= &
\frac{\mathcal E\left(\left(\widehat \pi + \Delta x\gamma^{0,n} + \varepsilon\gamma^{1,n} \right) \cdot  R \right)}{\mathcal E((\varepsilon\theta)\cdot R)}\in\mathcal X(1,\varepsilon),\\
\end{split}
\end{equation}
 by Lemma \ref{keyConcreteCharLem} and where the third line in \eqref{8241} is exactly $\frac{\tilde X^{\Delta x, \varepsilon, n}}{x + \Delta x}$ from \eqref{8122}. 
Note that in \eqref{8172}, we used the
Sherman-Morrison inversion formula, which asserts that
}
\textcolor{black}{
 $$\left( I - \varepsilon \vec 1 \theta^{\top}\right)^{-1} = I + \frac{\varepsilon \vec 1 \theta^{\top}}{1 - \varepsilon \vec 1^{\top} \theta}=I + \frac{\varepsilon }{1 - \varepsilon }\vec 1 \theta^{\top},$$
 where in the last equality, we have used again $\vec 1^{\top} \theta  \equiv 1$. 
Therefore, the invertibility of  $\left( I - \varepsilon \vec 1 \theta^{\top}\right)$ holds if and only if 
$\varepsilon\neq 1$. Thus,  in view of \eqref{8241}, the processes in \eqref{8172} match the indirect utility up to the second order in the sense \eqref{1121}. 
Now, in \eqref{8172} the integrand can be rewritten as follows.
 \begin{displaymath}
\begin{split}
&\left(\left(\widehat \pi + \Delta x\gamma^{0,n} + \varepsilon(-\theta + \gamma^{1,n}) \right)^\top\left((1 - \varepsilon)I + \varepsilon\vec 1 \theta^{\top}\right)\right)^{\top} \\
=& \left((1 - \varepsilon)I + \varepsilon \theta\vec 1^{\top}\right)\left(\widehat \pi + \Delta x\gamma^{0,n} + \varepsilon(-\theta + \gamma^{1,n}) \right).
\end{split}
\end{displaymath} 
The latter expression coincides with the one in \eqref{01134},  and, in view of \eqref{8172}, these are the proportions invested in traded assets under the num\'eraire $N^\varepsilon$.}
\end{proof}



\subsection{\textcolor{black}{On perturbations of models that admit closed-form solutions}}\textcolor{black}{
There are many models that admit explicit solutions, see  \cite{Thaleia01}, \cite{KallsenLog}, \cite{imHuMuller05},  \cite{KS2006b}, \cite{GR12}, \cite{Horst2014}, and \cite{Santacroce2} for their constructions and characterizations. In most cases, these solutions depend heavily on the exact dynamics of the stock price, and such solutions cease to exist under perturbations of the model parameters. The results of this paper provide both a stability result (as Theorems \ref{mainThm1} and \ref{161} assert that the value functions and the optimizers of the perturbed models are close to the ones of the unperturbed models) and a  constructive way of obtaining nearly optimal wealth processes and strategies.}

In the preferences are given by power utilities, then closed-form solutions are obtained in, e.g., \cite{GR12}, among others. In the asymptotic analysis, the corrections associated with perturbations of the initial wealth are trivial, as we have
\begin{displaymath}
\widehat X(x,0) = x\widehat X(1,0)  = x\mathcal E\left(\widehat \pi \cdot R\right).
\end{displaymath}
Thus, for the power utility case, in \eqref{8122}, only $\gamma^{1,n}$'s have to be estimated as $\gamma^{0,n} \equiv 0$. The Kunita-Watanabe decomposition provides a characterization of $\gamma^{1,n}$, as the risk-tolerance wealth process exists for the power utility and it is equal to $\widehat X(1,0)$ up to a multiplicative constant. Therefore, the measures $\mathbb {\widetilde R}$ and $\mathbb R$ coincide. This, in particular, is implicitly used in \cite{LMZ}, in the context of perturbations of the market price of risk.

\textcolor{black}{ 
In the case of general utility functions satisfying Assumption \ref{rra}, models that admit closed-form or fairly explicit solutions, are also studied, see,  e.g., \cite{KS2006b} and \cite{MostovyiSirbuModel}. By \cite[Theorem 6]{KS2006b}, 
a  class of models that gives the existence of the risk-tolerance wealth process for every utility function satisfying Assumption \ref{rra}  is the one, where the dual domain $\mathcal Y(1,0)$ admits a maximal element in the sense of the second-order stochastic dominance, i.e., an element $\widehat Y\in\mathcal Y(1,0)$, such that for every $Y\in\mathcal Y(1,0)$, we have
$$\int_0^z\mathbb P[\widehat Y_T\geq y]dy \geq \int_0^z\mathbb P[Y_T\geq y]dy,\quad z\geq0.$$
For example, this holds in a market, where there is a bank account with $0$ interest rate and only one traded stock, whose return is given by:
$$\rho^1_t = \mu t + \sigma B_t,\quad t\in[0,T],$$
for some constants $\mu$ and $\sigma>0$, 
 where the filtration is generated by $(B,W)$ a two-dimensional Brownian motion. Let us consider a one-dimensional and $\rho^1$-integrable process $\theta^1$, such that Assumption \ref{integrabilityAssumption} holds for  $\bar R = -\theta^1\cdot \rho^1 = -\theta\cdot R$, where $\theta = \begin{pmatrix} 1 - \theta^1\\ \theta^1\end{pmatrix}$. 
 In this case, the corresponding family of num\'eraires is $$N^\varepsilon = \mathcal E(\varepsilon\theta^1 \cdot \rho^1),\quad \varepsilon \in\mathbb R.$$
 Here $\varepsilon_0$ from Assumption \ref{boundedJumps} can be set to $\infty$, as there are no jumps of the underlying process $\rho^1$. 
 For a given $x>0$, let us consider $\widehat \pi^1$ and $\pi^{\mathcal R,1}$, such that $\widehat X(x,0) = x\mathcal E\left(\widehat \pi^1\cdot \rho^1 \right)$ and $\mathcal R(x) = \mathcal R_0(x) \mathcal E\left(\pi^{\mathcal R,1}\cdot \rho^1\right)$. Here,  both $\widehat \pi^1$ and $\pi^{\mathcal R,1}$ can be written in terms of the solution to a heat equation. 
Using an $\mathbb {\widetilde R}(x)$ local martingale $R^{\mathcal R,1}:= \rho^1 - \pi^{\mathcal R,1}\cdot [ \rho^1]$ and following Theorem \ref{riskTolThm2}, one needs to consider \eqref{eq:Kunita-Watanabe}, which gives the decomposition of the process $P$, and which in the present settings becomes 
 \begin{equation}\nonumber P=P_0- \varphi\cdot  R^{\mathcal R,1} -
 \varphi^{\perp} \cdot W,
 \end{equation}
 for some processes $\varphi$ and $\varphi^{\perp}$.
%
%
%
%
With
 \begin{equation}\nonumber
 \begin{split}
\zeta^1_t&:= \frac{\mathcal R_t(x)}{\widehat X_t(x,0)\mathcal R_0(x)}(\pi^{\mathcal R,1}_t - \widehat\pi^1_t),\\
\zeta^{0}_t &: = 1 - \zeta^1_t,\\
\upsilon^1_t& := \frac{\mathcal R_t(x)}{\widehat X_t(x,0)}\left(\left({\varphi}\cdot  R^{\mathcal R, 1}\right)_t(\pi^{\mathcal R,1}_t - {\widehat\pi^1_t}) + {\varphi_t}\right)\frac{1}{x},\\
\upsilon^0_t&: = 1 - \upsilon^0_t,\quad t\in[0,T],
\end{split}
\end{equation} 
and  by defining $$\gamma^0 := \begin{pmatrix}\zeta^{0} \\ \zeta^1\end{pmatrix}\quad {\rm and} \quad\gamma^1 := \begin{pmatrix}\upsilon^{0} \\ \upsilon^1\end{pmatrix},$$
one can construct $\gamma^{i, n}$, $i = 0,1$ and $n\in\mathbb N$, appearing in \eqref{8122} via setting $\gamma^{0,n} = \gamma^01_{[0,\tau_n]}$ and $\gamma^{1,n} = \gamma^11_{[0,\sigma_n]}$, $n\in\mathbb N$, where $\tau_n$, $n\in\mathbb N$,  is a localizing sequence for both $M^0(x,0)$ and $[ M^0(x,0)]$ and $\sigma_n$, $n\in\mathbb N$, is a localizing sequence for both $M^1(x,0)$ and $[ M^1(x,0)]$. Note that to get further characterizations of $\gamma^{1,n}$, one typically needs  $\theta$ to be chosen in a more explicit (and restrictive) form that admits a characterization of $\varphi$ in terms of  a system of ordinary differential equations in the spirit of \cite[Example 5.3]{LMZ}. 
Then, with such $\gamma^{i,n}$'s, the nearly optimal wealth processes are given by \eqref{8122}, which reads
\begin{equation}\nonumber
\begin{split}
\tilde X^{\Delta x, \varepsilon, n} &= (x + \Delta x) \mathcal E\left(\left( {\widehat{\pi}} + \Delta x\gamma^{0,n} + \varepsilon(-\theta + \gamma^{1,n}) \right)\cdot R^{\{\varepsilon\theta\}}\right),
\end{split}
\end{equation}
and where  $R^{\{\varepsilon\theta\}}$ is specified in \eqref{Reps}  that in the current settings becomes
$$R^{\{\varepsilon\theta\}} = \begin{pmatrix}0\\ \rho^1 - \varepsilon\theta^1\cdot [ \rho^1] \end{pmatrix}.$$
Therefore, we can rewrite the expression for $\tilde X^{\Delta x, \varepsilon, n}$ as 
$$ (x + \Delta x) \mathcal E\left(\left( {\widehat{\pi}^1} + \Delta x\zeta^{1}1_{[0,\tau_n]} + \varepsilon(-\theta^1 + \upsilon^{1}1_{[0,\sigma_n]}) \right)\cdot \left(\rho^1 - \varepsilon\theta^1\cdot [ \rho^1]\right)\right),$$
where $\widehat \pi^1$ is the second component of $\widehat \pi$. 
Note that for the wealth process $\tilde X^{\Delta x, \varepsilon, n}$,  the proportions of the capital invested in the bank account and stock under the num\'eraire $N^\varepsilon$ are given by \eqref{01134}, which in the current settings reads
\begin{equation}\label{1152}
\begin{pmatrix}
(1-\varepsilon) \left(1 - \left({\widehat{\pi}^1} + \Delta x\zeta^{1}1_{[0,\tau_n]} + \varepsilon(-\theta^1 + \upsilon^{1}1_{[0,\sigma_n]})\right)\right) + \varepsilon(1- \theta^1) \\ 
(1 - \varepsilon)\left({\widehat{\pi}^1} + \Delta x\zeta^{1}1_{[0,\tau_n]} + \varepsilon(-\theta^1 + \upsilon^{1}1_{[0,\sigma_n]})\right) + \varepsilon \theta^1
\end{pmatrix}.
\end{equation}
Further, with
$$\widetilde \pi^{\{1, \Delta x, \varepsilon, n\}}: =  {\widehat{\pi}^1} + \Delta x\zeta^{1}1_{[0,\tau_n]} + \varepsilon(-\theta^1 + \upsilon^{1}1_{[0,\sigma_n]}),$$
one can rewrite \eqref{1152} as
\begin{equation}\label{1151}
\begin{pmatrix}
(1-\varepsilon) \left(1 - \widetilde \pi^{\{1, \Delta x, \varepsilon, n\}}\right) + \varepsilon(1- \theta^1) \\ 
(1 - \varepsilon)\widetilde \pi^{\{1, \Delta x, \varepsilon, n\}} + \varepsilon \theta^1
\end{pmatrix}.
\end{equation}
To recapitulate, in the context of the stochastically dominant model specified above, \eqref{1151} gives proportions invested in the traded assets under the (perturbed) num\'eraires $N^\varepsilon = \mathcal E(\varepsilon\theta^1\cdot \rho^1)$'s, such that the corresponding wealth processes $\tilde X^{\Delta x, \varepsilon, n}$'s  match the indirect utility up to the second order in the sense of Theorem \ref{thmCorOptimizer}, see \eqref{1121} in the statement of this theorem.}
\subsection{On an alternative parametrization of perturbations and a relation to perturbations of the drift and/or volatility}\label{8241a}
In view of the family $R^{\{\varepsilon\theta\}}$, $\varepsilon\in(-\varepsilon_0,\varepsilon_0)$ defined in \eqref{Reps}, that drive the processes \eqref{8122}, a different type of parametrization of perturbations of the form \eqref{Reps} can be used. We will illustrate this in the settings, where $R$ is {\it continuous}. In this case,
if $\theta$ is of the form $-\psi e^i$, where $\psi$ is a one-dimensional bounded and predictable, $\left(\sum\limits_{j = 0}^d[\rho^i, \rho^j]\right)$-integrable process, and $e^i$ is a (constant-valued) vector whose $i$-th component equals to $1$ and all other components equal to zero,
we the following dynamics of the returns of the stocks for perturbed models:
\begin{equation}\label{8191}
\begin{split}
R^{\varepsilon, j} &=\rho^j,\quad {\rm if}\quad  j\neq i,\\
R^{\varepsilon, j} &=\rho^j + \varepsilon \psi \cdot\left(\sum\limits_{k = 0}^d [\rho^k, \rho^j]\right),\quad {\rm if}\quad j=i,
\end{split}
\end{equation}
which in turn corresponds to perturbations of the finite-variation part of the $i$-th asset return only.
This allows to consider perturbations of the finite-variation part of the return process. 
Moreover, by a different choice of $\theta$, we can achieve simultaneous perturbations of multiple returns. 

The relationship between these parametrization and the one considered in the remaining part of the paper can be obtained following the argument in the proof of Theorem \ref{thmCorOptimizer}, see \eqref{1251} there.  
Thus, for perturbations of the form \eqref{8191}, under appropriate regularity conditions (similar to the ones in Theorem \ref{mainThm1}), the expansions of the value functions, derivatives of the optimal wealth processes, and approximations of trading strategies of the form \eqref{8122} follow from the results of the present~paper.

Let us discuss the relation to the framework in \cite{MostovyiSirbuModel}, where there is one traded stock, whose return, $\rho^1$, follows
$$\rho^1 = M + \lambda \cdot\langle M\rangle,$$
where $M$ is a continuous local martingale. 
In this case, \eqref{8191} gives the following dynamics for the perturbed models
\begin{equation}\nonumber
\begin{split}
R^{\varepsilon, 1} &=  \rho^1  + \varepsilon \psi \cdot\langle M\rangle \\
&=  M  + (\lambda + \varepsilon \psi) \cdot\langle M\rangle,
\end{split}
\end{equation}
which is the parametrization of perturbations in \cite{MostovyiSirbuModel}. Further, the prototypical wealth process for a perturbed model, for some $\pi$, is given by 
\begin{equation}\nonumber
\begin{array}{rcl}
X^{\varepsilon} &= &x \mathcal E\left(\pi\cdot\left(\rho^1  + \varepsilon \psi \cdot\langle M\rangle\right)\right).\\
\end{array}
\end{equation}
Under the appropriate boundedness of $\psi$, with $\bar \pi := \pi(\lambda + \varepsilon \psi)$, the evolution of $X^\varepsilon$ can be rewritten as
$$X^{\varepsilon} = x \mathcal E\left((\pi(\lambda + \varepsilon \psi))\cdot \left(\lambda \cdot \langle M\rangle + \frac{\lambda}{\lambda + \varepsilon \psi}\cdot M \right)\right)
=x \mathcal E\left( \bar \pi\cdot\left(\lambda \cdot\langle M\rangle + \frac{\lambda}{\lambda + \varepsilon \psi}\cdot M \right)\right).$$
This corresponds to perturbations of the martingale part (or volatility) of the return, similar to the ones in \cite{herJohSei17}.

\section{Counterxamples}\label{secCounterexamples}
The following example demonstrates the necessity of Assumption \ref{integrabilityAssumption}.
\begin{exa}\label{8231}
Let us assume that the market consists of a bond with zero interest rate and one stock with return $B$, where $B$ is a Brownian motion on the filtered probability space $\left(\Omega, \mathcal F, \left( \mathcal F_t\right)_{t\in[0,1]}, \mathbb P\right)$, where $1$ is the time horizon and $(\mathcal F_t)_{t\in[0,1]}$ is the usual augmentation of the filtration generated by $B$. In this case $\mathbb P$ is the martingale measure. Let us also suppose that $U(x) = \frac{x^p}{p},$ $x\in(0,\infty)$, where $p\in(0,1)$. An application of Jensen's inequality implies that 
for every $y>0$, $v(y) = V(y) = \frac{y^{-q}}{q}$, where $q = \frac{p}{1-p}$, and (a constant-valued process) $y$ is the dual minimizer.

For the perturbed models, where $\bar R = -\theta\cdot B$ is such that $\bar R_1 = |B_1|^{2+ \delta}\sign(B_1)$ for some $\delta >0$. Then, $\mathbb R(x) = \mathbb P$, $x>0$, and for every constant $c>0$, we have 
\begin{displaymath}
\begin{split}
\mathbb E^{\mathbb R(x)}\left[ \exp\left(c(|\bar R_1| + [\bar R,\bar R]_1) \right)\right] &\geq  
\mathbb E\left[ \exp\left(c|B_1|^{2+ \delta}\sign(B_1)\right)\right] \\
&=   \tfrac{1}{\sqrt{2\pi}}\int_{\mathbb R} \exp\left(c|y|^{2+ \delta}\sign(y) - \tfrac{1}{2}y^2\right)dy\\
&= \infty,
\end{split}
\end{displaymath} 
i.e., Assumption \ref{integrabilityAssumption} does not hold. 
Nevertheless, $N^\varepsilon = \mathcal E\left(-\varepsilon\bar R\right)$ is a strictly positive wealth process for every $\varepsilon\in\mathbb R$ and thus a num\'eraire.
For every $x>0$ and $\varepsilon\neq 0$, we have
\begin{displaymath}
\begin{split}
u(x,\varepsilon) &\geq  \mathbb E\left[U\left(\frac{x}{N^\varepsilon_1}\right) \right]\\
&= \mathbb E\left[U\left({x}\exp\left(\varepsilon\bar R_1+\tfrac{\varepsilon^2}{2}[\bar R,\bar R]_1 \right)\right) \right]\\
&\geq \frac{x^p}{p}\mathbb E\left[\exp\left(\varepsilon p \bar R_1\right)\right]\\
&=  \frac{x^p}{p}\mathbb E\left[\exp\left(\varepsilon p |B_1|^{2+ \delta}\sign(B_1)\right)\right]\\
&=  \frac{x^p}{p\sqrt{2\pi}}\int_{\mathbb R}\exp\left(\varepsilon p |y|^{2+ \delta}\sign(y) -\tfrac{1}{2}y^2\right)dy\\
&=  \infty.
\end{split}
\end{displaymath}
\end{exa}
The following example shows that without Assumption \ref{boundedJumps}, we might have a family of processes $(N^\varepsilon)_{\varepsilon \in (-\varepsilon_0, \varepsilon_0)}$, such that for every $\varepsilon\neq 0$, $N^\varepsilon_T < 0$ with positive probability.
\begin{exa}\label{counterExBoundedJumps}
Let us consider model, where there are three times: $0$, $1$, and $2$, where the process ${R}$ is a one-dimensional semimartingale such that 
$${R}_0 = {R}_1 = 1,~ \Pas,~
{\rm and}~{R}_2~{\rm equals~to}~3/2~{\rm or}~1/2~{\rm with~probability}~1/2~{\rm each}.$$ Let us also consider a predictable process $\theta$, such that $$
\theta_1=0,~\Pas,~\theta_2 = n~{\rm with~probability}~\tfrac{1}{2^n},~n\in\mathbb N.$$  Then in \eqref{numReturn}, for every $\varepsilon\neq 0$,
$$\mathbb P\left[\Delta \left( (\varepsilon\theta)\cdot R\right)_2 <-1\right]= \mathbb P\left[
\varepsilon\theta_2( R_2 - R_1) <-1\right] >0,$$
thus, $N^\varepsilon_2<0$ with positive probability. Therefore, {\it for every $\varepsilon\neq 0$, $N^\varepsilon$ is not a num\'eraire}.
\end{exa}
\subsection*{On the necessity of the remaining assumptions}

\begin{enumerate}
\item Conditions \eqref{finCond} and  \eqref{NUPBR} are necessary for the expected utility maximization problem to admit standard conclusions of the utility maximization theory, see the abstract theorems in \cite{KS} 
 and  \cite[Proposition 4.19]{KarKar07}.
We stress that  \eqref{finCond} and  \eqref{NUPBR} are only imposed for $\varepsilon = 0$.

\item Modeling the evolution of stocks with semimartingales is necessary for the absence of arbitrage as above, see \cite[Theorem~1.3]{KostasEmery}, see also \cite[Theorem 1.3]{KostasPlatenSem} for the case of the nonnegative stock price process.

\item If sigma-boundedness in the sense of Assumption \ref{sigmaBoundedness} does not hold, then the second-order expansion in the initial wealth might not exist, see \cite[Example~3]{KS2006}.

\item 
\cite[Example 1 and Example  2]{KS2006} show the necessity of Assumption \ref{rra} for two-times differentiability of the value function in $x$. Note that, by the concavity of the value function in the $x$ variable, two-times differentiability in the $x$ variable at $x>0$ holds if and only if the value function admits a quadratic expansion at $x$ (in the $x$ variable), see \cite[Theorem 5.1.2]{LemHurMinAlg}.
\end{enumerate}
\bibliographystyle{alpha}\bibliography{literature1}

\newcommand{\etalchar}[1]{$^{#1}$}
\begin{thebibliography}{HMKS17}

\bibitem[Bec01]{Becherer}
D.~Becherer.
\newblock The num{\'e}raire portfolio for unbounded semimartingales.
\newblock {\em Finance Stoch.}, 5:327--341, 2001.

\bibitem[BS12]{MihaiSara}
S.~Biagini and M.~S{\^i}rbu.
\newblock A note on admissibility when the credit line is infinite.
\newblock {\em Stochastics}, 84(2-3):157--169, 2012.

\bibitem[{\v C}K07]{KalCher07}
A.~{\v C}ern{\' y} and J.~Kallsen.
\newblock On the structure of general mean-variance hedging strategies.
\newblock {\em Ann. Probab.}, 35(4):1479--1531, 2007.

\bibitem[CLP98]{Pham98}
Gourieroux C., J.P. Laurent, and H.~Pham.
\newblock Mean-variance hedging and num{\'e}raire.
\newblock {\em Math. Finance}, 8:179--200, 1998.

\bibitem[CS13]{CzichSchweiz2013}
C.~Czichowsky and M.~Schweizer.
\newblock Cone-constrained continuous-time {M}arkowitz problems.
\newblock {\em Ann. Appl. Probab.}, 23(2):764--810, 2013.

\bibitem[GEKR95]{GemKarRoc95}
H.~Geman, N.~El~Karoui, and J-C. Rochet.
\newblock Changes of num{\'e}raire, changes of probability and option pricing.
\newblock {\em J. Appl. Probab.}, 32:443--458, 1995.

\bibitem[GK03]{KallsenLog}
T.~Goll and J.~Kallsen.
\newblock A complete explicit solution to the log-optimal portfolio problem.
\newblock {\em Ann. Appl. Probab.}, 13:774--799, 2003.

\bibitem[GR12]{GR12}
P.~Guasoni and S.~Robertson.
\newblock Portfolios and risk premia for the long run.
\newblock {\em Ann. Appl. Probab.}, 22(1):239--284, 2012.

\bibitem[Hen02]{vicky02}
V.~Henderson.
\newblock Valuation of claims on nontraded assets using utility maximization.
\newblock {\em Math. Finance}, 12:351--373, 2002.

\bibitem[HH02]{henHob02}
V.~Henderson and D.~Hobson.
\newblock Real options with constant relative risk aversion.
\newblock {\em J. Econom. Dynam. Control}, 27:329--355, 2002.

\bibitem[HH09]{HobHen}
V.~Henderson and D.~Hobson.
\newblock Utility indifference pricing: An overview.
\newblock 2009.
\newblock Indifference Pricing: Theory and Applications Editor: R. Carmona,
  Princeton University Press. February 2009.

\bibitem[HHI{\etalchar{+}}14]{Horst2014}
U.~Horst, Y.~Hu, P.~Imkeller, A.~R\'eveillac, and J.~Zhang.
\newblock Forward-backward systems for expected utility maximization.
\newblock {\em Stoch. Process. Appl.}, 124(5):1813 -- 1848, 2014.

\bibitem[HIM05]{imHuMuller05}
Y.~Hu, P.~Imkeller, and M.~M{\"u}ller.
\newblock Utility maximization in incomplete markets.
\newblock {\em Ann. Appl. Probab.}, 15(3):1691--1712, 2005.

\bibitem[HMKS17]{herJohSei17}
S.~Herrmann, J.~Muhle-Karbe, and F.~T. Seifried.
\newblock Hedging with small uncertainty aversion.
\newblock {\em Finance Stoch.}, 21:1--64, 2017.

\bibitem[HUL96]{LemHurMinAlg}
J.-B. Hiriart-Urruty and C.~Lemarechal.
\newblock {\em Convex Analysis and Minimization Algorithms}.
\newblock Springer, 2nd edition, 1996.

\bibitem[JMSS12]{Santacroce1}
M.~Jeanblanc, M.~Mania, M.~Santacroce, and M.~Schweizer.
\newblock Mean-variance hedging via stochastic control and {BSDE}s for general
  semimartingales.
\newblock {\em Ann. Appl. Probab.}, 22(6):2388--2428, 2012.

\bibitem[JS03]{JS}
J.~Jacod and A.~N. Shiryaev.
\newblock {\em Limit Theorems for Stochastic Processes}.
\newblock Springer, 2003.
\newblock 2nd edition.

\bibitem[Kal02]{Kallsen02}
J.~Kallsen.
\newblock Derivative pricing based on local utility maximization.
\newblock {\em Finance Stoch.}, 6:115--140, 2002.

\bibitem[Kar13]{KostasEmery}
K.~Kardaras.
\newblock On the closure in the {E}mery topology of semimartingale
  wealth-process sets.
\newblock {\em Ann. Appl. Probab.}, 23(4):1355--1376, 2013.

\bibitem[KK07]{KarKar07}
I.~Karatzas and K.~Kardaras.
\newblock The num\'{e}raire portfolio in semimartingale financial models.
\newblock {\em Finance Stoch.}, 11:447--493, 2007.

\bibitem[KKS16]{KabKarSong16}
Y.~Kabanov, K.~Kardaras, and S.~Song.
\newblock No arbitrage of the first kind and local martingale num{\'e}raires.
\newblock {\em Finance Stoch.}, 20(4):1097--1108, 2016.

\bibitem[KP11]{KostasPlatenSem}
K.~Kardaras and E.~Platen.
\newblock On the semimartingale property of discounted asset-price processes.
\newblock {\em Stochastic Process. Appl.}, 121(11):2678--2691, 2011.

\bibitem[KS99]{KS}
D.~Kramkov and W.~Schachermayer.
\newblock The asymptotic elasticity of utility functions and optimal investment
  in incomplete markets.
\newblock {\em Ann. Appl. Probab.}, 9(3):904--950, 1999.

\bibitem[KS06a]{KS2006}
D.~Kramkov and M.~S\^irbu.
\newblock On the two-times differentiability of the value functions in the
  problem of optimal investment in incomplete markets.
\newblock {\em Ann. Appl. Probab.}, 16(3):1352--1384, 2006.

\bibitem[KS06b]{KS2006b}
D.~Kramkov and M.~S\^irbu.
\newblock Sensitivity analysis of utility-based prices and risk-tolerance
  wealth process.
\newblock {\em Ann. Appl. Probab.}, 16(4):2140--2194, 2006.

\bibitem[LM{\v{Z}}18]{LMZ}
K.~Larsen, O.~Mostovyi, and G.~{\v{Z}}itkovi{\'c}.
\newblock An expansion in the model space in the context of utility
  maximization.
\newblock {\em Finance Stoch.}, 22(2):297--326, 2018.

\bibitem[LP99]{laurenPham}
J.P. Laurent and H.~Pham.
\newblock Dynamic programming and mean-variance hedging.
\newblock {\em Finance Stoch.}, 3:83--110, 1999.

\bibitem[Mon13]{Monoy2013}
M.~Monoyios.
\newblock Malliavin calculus method for asymptotic expansion of dual control
  problems.
\newblock {\em SIAM J. Fin. Math.}, 4:884--915, 2013.

\bibitem[MS19]{MostovyiSirbuModel}
O.~Mostovyi and M.~S{\^i}rbu.
\newblock Sensitivity analysis of the utility maximization problem with respect
  to model perturbations.
\newblock {\em Finance Stoch.}, 23:1--46, 2019.
\newblock published online.

\bibitem[Pha09]{PhamBook}
H.~Pham.
\newblock {\em Continuous-time Stochastic Control and Optimization with
  Financial Applications}, volume~61 of {\em Stochastic Modelling and Applied
  Probability}.
\newblock Springer, Berlin, 2009.

\bibitem[PRS98]{phamSchw}
H.~Pham, T.~Rheinlander, and M.~Schweizer.
\newblock Mean-variance hedging for continuous processes: new proofs and
  examples.
\newblock {\em Finance Stoch.}, 2:173--198, 1998.

\bibitem[Rob17]{Scott17}
S.~Robertson.
\newblock Pricing for large positions in contingent claims.
\newblock {\em Math. Finance}, 27(3):746--778, 2017.

\bibitem[RSA17]{RSA17}
S.~Robertson, K.~Spiliopoulos, and M.~Anthropelos.
\newblock The pricing of contingent claims and optimal positions in
  asymptotically complete markets.
\newblock {\em Ann. Appl. Probab.}, 27(3):1778--1830, 2017.

\bibitem[ST14]{Santacroce2}
M.~Santacroce and B.~Trivellato.
\newblock Forward backward semimartingale systems for utility maximization.
\newblock {\em SIAM J. Control Optim.}, 52(6):3517--3537, 2014.

\bibitem[TS14]{TakaokaSchweizer}
K.~Takaoka and M.~Schweizer.
\newblock A note on the condition of no unbounded profit with bounded risk.
\newblock {\em Finance Stoch.}, 18(2):393--405, 2014.

\bibitem[Zar01]{Thaleia01}
T.~Zariphopoulou.
\newblock A solution approach to valuation with unhedgeable risks.
\newblock {\em Finance Stoch.}, 5:61--82, 2001.

\end{thebibliography}
\end{document}